\documentclass[english]{article}
\usepackage[numbers]{natbib}
\usepackage{geometry}
\usepackage[hyperindex,breaklinks,colorlinks,citecolor=blue,urlcolor=blue]{hyperref}

\usepackage{float}
\usepackage{mathrsfs}
\usepackage{amsthm}
\usepackage{amsfonts}
\usepackage{epsfig}
\usepackage{bm}
\usepackage{amsmath}
\numberwithin{equation}{section}
\usepackage{amssymb}
\usepackage{graphicx}
\usepackage{caption}
\usepackage{bbm}
\usepackage{subcaption}

\numberwithin{equation}{section}
\newtheorem{theorem}{Theorem}[section]
\newtheorem{thm}{Theorem}[section]
\newtheorem{lemma}[thm]{Lemma}
\newtheorem{proposition}[thm]{Proposition}

\newtheorem{remark}[thm]{Remark}

\newcommand{\LL}{\mathcal{L}}
\newcommand{\B}{\mathcal{B}}

\newcommand{\F}{\mathcal{F}}

\newcommand{\OO}{\mathcal{O}}

\newcommand{\Normal}[1]{\mathcal{N}\BK{ #1 }}

\newcommand{\E}{\mathbb{E}}

\newcommand{\PP}{\mathbb{P}}

\newcommand{\bra}[1]{\langle #1 \rangle}
\newcommand{\BK}[1]{ {\left( #1 \right)} }
\newcommand{\sqBK}[1]{ {\left[ #1 \right]} }
\newcommand{\curBK}[1]{ {\left\{ #1 \right\}} }

\newcommand{\norm}[1]{\left\Vert #1 \right\Vert}

\newcommand{\SB}{\mathrm{S}_{\B}}

\newcommand{\I}{\mathcal{I}}

\newcommand{\II}{\mathbbm{1}}
\newcommand{\ol}{\overline}

\newcommand{\R}{\mathbb{R}}

\newcommand{\Z}{\mathbb{Z}}

\newcommand{\N}{\mathbb{N}}

\newcommand{\Var}{\mathrm{Var}}

\newcommand{\Cesc}{C_{\mathrm{esc}}}

\newcommand{\mtx}[1]{\bm{#1}}
\newcommand{\dtv}{d_{\mathrm{TV}}}

\newcommand{\tbeta}{\tilde{\beta}}
\newcommand{\tbetac}{\tilde{\beta}_c}
\newcommand{\ts}{\tilde{s}}
\newcommand{\const}{\mathrm{const}}
\newcommand{\defby}{:=}
\newcommand{\vertiii}[1]{{\left\vert\kern-0.25ex\left\vert\kern-0.25ex\left\vert #1 
    \right\vert\kern-0.25ex\right\vert\kern-0.25ex\right\vert}}

\newcommand{\GG}{\bm{G}}

\newcommand{\KK}{\bm{K}}
\newcommand{\bmL}{\bm{L}}
\renewcommand{\P}{\bm{P}}
\newcommand{\ppi}{\bm{\pi}}
\newcommand{\mmu}{\bm{\mu}}
\newcommand{\meta}{\bm{\eta}}

\newcommand{\wtilde}{\widetilde}
\newcommand{\tsigma}{\wtilde{\sigma}}

\renewcommand{\phi}{\varphi}
\renewcommand{\epsilon}{\varepsilon}

\begin{document}

\begin{center}
{\Large \textbf{Error Bounds for Sequential Monte Carlo Samplers for Multimodal Distributions}}

\vspace{0.5cm}

BY DANIEL PAULIN, AJAY JASRA \& ALEXANDRE THIERY

{\footnotesize Department of Statistics \& Applied Probability, \\
National University of Singapore, 6 Science Drive 2, Singapore, 117546, SG.}\\
{\footnotesize E-Mail:\,}\texttt{\emph{\footnotesize paulindani@gmail.com, staja@nus.edu.sg, a.h.thiery@nus.edu.sg}}\\

\vspace{0.35cm}
\today

\begin{abstract}
In this paper, we provide bounds on the asymptotic variance for a class of sequential Monte Carlo (SMC) samplers designed for approximating multimodal distributions. 
Such methods combine standard SMC methods and Markov chain Monte Carlo (MCMC) kernels.
Our bounds improve upon previous results, and unlike some earlier work, they  also apply in the case when the MCMC kernels can move between the modes. We apply our results to the Potts model from statistical physics. In this case, the problem of sharp peaks is encountered. Earlier methods, such as parallel tempering, are only able to sample from it at an exponential (in an important parameter of the model) cost. We propose a sequence of interpolating distributions called \emph{interpolation to independence}, and show that the SMC sampler based on it is able to sample from this target distribution at a polynomial cost. We believe that our method is generally applicable to many other distributions as well.
\end{abstract}

\end{center}

\section{Introduction}
Sequential Monte Carlo sampling \citep{DDJJRRSB,JNE,NAIS} is a method designed to approximate a sequence of probability distributions $\{\mmu_{k}\}_{0\leq k \leq n}$ defined upon a common measurable space $(E,\mathcal{E})$. The method uses $N\geq 1$ samples (or particles) that are generated in parallel and propagated via importance sampling and resampling methods. 
In the context of this article, we are concerned with the class of algorithms where $\mmu_0$ is an easy to sample distribution and $\mmu_n$ is a complex distribution and $\mmu_1,\dots,\mmu_{n-1}$ interpolate (in some sense) between $\mmu_0$ and $\mmu_n$. In addition, the particles are moved/mutated through Markov kernels of invariant measure $\mmu_k$ at time $k$.
The SMC methodology has proven itself to be a very efficient tools for inference in a wide variety of statistical models and applications including stochastic volatility \citep{JSDTLD}, regression models \citep{CIBIS} and approximate Bayesian computation \citep{DDJABC}.
In this article, we develop theoretical tools for analysing SMC samplers \citep{DDJJRRSB}, a particular class of SMC algorithms, and introduce a new type of interpolating sequences of distributions that enjoys, in many situations, better convergence properties than more standard tempering sequences that are usually used in practice.
%SMC samplers are a certain class of SMC algorithms and estimates 
%of expectations w.r.t.~$\{\mmu_{k}\}_{0\leq k \leq n}$  can be shown 
%to be consistent, with a central limit theorem (as $N\rightarrow\infty$) 
%using standard theory in, for example, \citep{del2013mean}. 
%In addition, 
The SMC methodology is by now fairly well understood; for example, high-dimensional asymptotic results are obtained in \citep{BCJHD,BCJWHD}, the study of the long-time behaviour is presented in \citep{JDSSMC,WSMCS}
and its performances for exploring multimodal distributions are described in \citep{EberleMartinelliPTRF,schweizer2012non}. For a book-length treatment of the subject, the reader is referred to \citep{del2013mean}. \cite{nonasymptoticDelMoral} has shown concentration inequalities and moment bounds  that take into account global mixing properties of the Markov kernels. These results are formulated in terms of the so-called Dobrushin coefficients, i.e. the contraction rate of the Markov kernels in total variational distance. The authors also generalise their results to stochastic optimization algorithms. The main results in this paper, Theorems \ref{thmuni} and \ref{thmwithmixing1} only bound the asymptotic variance. Theorem \ref{thmuni} is using global mixing properties (spectral gap) of the Markov kernels, while Theorem \ref{thmwithmixing1} is using local mixing properties (thus it is more suited to multimodal distributions). Proving non-asymptotic bounds under similar conditions is an interesting problem for further research.

Multimodal distributions appear in a wide variety of applications in statistics, physics, economics and many more. However, sampling from such distributions is a  challenging problem. In the context of interest, one well known advantage of SMC samplers over traditional Markov Chain Monte Carlo (MCMC) methods is their ability to work relatively well for multimodal distributions. Although this phenomenon is known to practitioners, there have been only very few attempts to rigorously explain and quantify this behaviour. \citep{EberleMartinelliPTRF} and \citep{schweizer2012non} were the first to show error bounds (moment bounds) for SMC samplers when applied to explore multimodal distributions.
%, in particular, variance, and moment bounds. 
These results are extremely interesting from a conceptual perspective; unfortunately, the applications of these results require very stringent assumptions that are rarely met in practical scenarios of interest to practitioners. 
%the very strong assumption that the Markov (MCMC) kernels 
%do not mix between the modes, which is rarely satisfied in practice.
One of the main purposes of this article is to develop widely applicable tools for studying the asymptotic properties of SMC samplers when applied to probe multimodal distributions. To this end, we leverage a metastable approximation to obtain new bounds on the asymptotic variance of the SMC estimates; see  \cite{OlivieriLargedevMetastability} for a comprehensive monograph on metastability. These bounds show that if the time scale it takes for the Markov kernels to approximate a mixture of local equilibrium distributions sufficiently well is polynomial in some size parameter of the system, then the SMC sampler can sample from the multimodal target distribution in polynomial time. In other words, this shows that Markov kernels with good metastability properties can be leveraged to construct SMC samplers  that can explore multimodal target distributions in polynomial time.
%the precise understanding of the metastability properties of the Markov kernels %used in an SMC sampler algorithm can lead to polynomial time sampling algorithms %for complex multimodal target distributions.

We demonstrate the applicability of our results by analysing a model from statistical physics, the Potts model with three colours at critical temperature. Earlier methods, such as parallel tempering, are only able to sample configurations from the Potts model at an exponential cost \citep{BhatnagarRandall} when using standard tempering bridging distributions; this is mainly caused by the appearance of both wide and narrow peaks in the target distribution Indeed, \cite{woodard2009sufficient} have shown that, in general, parallel and simulated tempering using tempering distributions are torpidly mixing for such target distributions. The recent paper \cite{bhatnagar2015simulated} has introduced model specific interpolating distributions for the Potts model called \emph{entropy dampening distributions} and proven (Theorem $7.7$ of \cite{bhatnagar2015simulated}) that the simulated tempering algorithm mixes in polynomial time when using these distributions.

The other main contribution of this article is the introduction of a new general interpolating sequence of distributions, coined \emph{interpolation to independence sequence}. We rigorously prove that an SMC sampler utilizing this newly developed interpolating sequence can generate configurations of the Potts model at a computational cost that only scales cubicly in the system size; this improves improving upon the earlier polynomial rate obtained in \cite{bhatnagar2015simulated}. The \emph{interpolation to independence sequence}  is not model specific; we believe that it has a wide range of potential applications to many systems that display scale invariance properties. 

%We note that although in this example we work with discrete probability %distributions, our error bounds are stated and proven for SMC samplers on general %state spaces.

The paper is organised as follows. In Section \ref{SecPreliminaries}, we introduce the basic tools required, such as Feynman-Kac semigroups. In Section \ref{SecResults}, we state and prove our general asymptotic bounds. In Section \ref{secinterpolationindependence}, we introduce the interpolation to independence sequence of distributions. Section \ref{SecPotts} states our results for the Potts model, and Section \ref{SecProofPotts} contains the proofs of these results.
%Notations
%\subsection{Notations}
%For a probability distribution $\pi$, the standard scalar product in $L^2(\pi)$ is denoted by $\bra{\cdot, \cdot}_{\pi}$; the notation $\pi(\phi) \defby \bra{\phi, \vct{1}}_\pi$ denotes the expectation of the function $\phi$ under the probability $\pi$ and $\Var_\pi(\phi)$ denotes its variance. For a set $A$ and test function $\phi$, we sometimes use the shorthand $\pi(f ; A) \defby \int_A f(x) \, \pi(dx)$ to denote the expectation of $f$ with respect to $\pi$ when restricted to the set $A$.
%\todo{Notations subsection removed because these were already defined in Table 1}

%
%  NOTATIONS
%
\subsection{Notations}
For a function $\phi:E \to \mathbb{R}$, the supremum norm is written as $\|\phi\|_{\infty} = \sup_{x\in E} \, |\phi(x)|$.
Consider a probability measure $\mmu$ on $E$ and $\phi \in L^1(\mmu)$; we repeatedly use the shorthand notation $\mmu(\phi) = \int_E \phi(x) \, \mmu(dx)$ and write $\Var_{\mmu}(\phi)$ for the variance of $\phi$ under $\mmu$. The Hilbert space $L^2(\mmu)$ has scalar product 
\begin{align*}
\bra{f,g}_{\mmu} = \int_E f(x) \, g(x) \, \mmu(dx)
\end{align*}
and associated norm $\| \cdot \|_{L^2(\mmu)}$.
We sometimes identify $\mmu$ with the linear operator from $L^2(\mmu)$ to itself that maps the function $\phi$ to the constant function that equals $\mmu(\phi)$ everywhere. A Markov kernel $\KK$ that lets $\mmu$ invariant is identified with the linear operator $\KK: L^2(\mmu) \to L^2(\mmu)$ 
\begin{align*}
\KK \phi(x) = \int \KK(x, dy) \, \phi(y).
\end{align*}
For an operator $\bmL: L^2(\mmu) \to L^2(\mmu)$, its triple norm  equals
\begin{align*}
\vertiii{ \bmL }_{L^2(\mmu)}
\; = \; 
\sup \curBK{\|\bmL \, \phi\|_{L^2(\mmu)} \; : \; \mmu(\phi^2) \leq 1}.
\end{align*}
Similarly, the quantity $\vertiii{ \bmL }_{\infty}$ equals the supremum of $\|\bmL \phi\|_{\infty}$ over the set of test functions such that $\|\phi\|_{\infty} \leq 1$. The notation $\Normal{m,\sigma^2}$ designates the Gaussian distribution with mean $m$ and variance $\sigma^2$. We use the notation $A \sqcup B$ to denote the union of the disjoint subsets $A,B \subset E$. Finally, for a function $\phi:E \to \R$ and a subset $S \subset E$, the function $\phi_{|S}: S \to \R$ is the restriction of $\phi$ to $S$; for $x \in S$, we have $\phi_{|S}(x) = \phi(x)$.

\section{Preliminaries}\label{SecPreliminaries}
Suppose that we are interested in inference from some distribution $\mmu$ on some Polish state space $(E,\mathcal{E})$. We define an interpolating sequence $\mmu_0, \mmu_1, \dots,\mmu_n$ of distributions with $\mmu_n=\mmu$; the distribution $\mmu_0$ is chosen so that it is straightforward to generate independent samples from it. In this article, we assume that for any index $0\le k\le n-1$ the distribution $\mmu_{k+1}$ is absolutely continuous with respect to $\mmu_k$ and denote by $g_{k,k+1}$ the Radon-Nykodym derivative
\begin{align*}
g_{k,k+1} \; = \; \frac{d \mmu_{k+1}}{d \mmu_k}.
\end{align*}
We work under the standing assumption that these Radon-Nikodym derivatives are bounded and set
\begin{align*}
\Gamma_g 
\; = \; 
\max \curBK{ \| g_{k,k+1} \|_{\infty} \; : \; 0 \leq k \leq n-1 } < \infty.
\end{align*}
We will make extensive use of the linear operator $\GG_{k,k+1}: L^2(\mmu_{k+1}) \to L^2(\mmu_k)$ defined as
\begin{align} \label{eq.operator.G}
\GG_{k,k+1} \phi = g_{k,k+1} \, \phi.
\end{align}
For a test function $\phi:E\to \R$, our goal is to estimate 
the performances of the Sequential Monte Carlo (SMC) algorithm for estimating the expectation $\mmu(\phi)$. Recall that the SMC algorihm with $N$ particles proceeds as follows \cite{DDJJRRSB}. An initial set $\{\xi_{0}^{1},\ldots, \xi_0^{N}\}$ of $N$ i.i.d samples from the probability distribution $\mmu_0$ is generated. The empirical distribution
\begin{align*}
\mmu^N_0 = (1/N) \, \sum_{i=1}^N \bm{\delta}_{\xi_0^i},
\end{align*}
where $\bm{\delta}_a$ denotes the Dirac mass at $a \in E$,
is an approximation of $\mmu_0$. In order to produce a particle approximation of $\mmu$, the algorithm iterates {\it mutation} and {\it resampling} steps. Suppose that a particle approximation  $\mmu^N_k = (1/N) \, \sum_{i=1}^N \bm{\delta}_{\xi_k^i}$ has already been obtained. The mutation steps generates $N$  particles $\{ \tilde{\xi}^1_k, \ldots, \wtilde{\xi}^N_k \}$ distributed as $\wtilde{\xi}^i_k \sim \KK_k(\xi^i_k, dx)$ where $\KK_k(x, dy)$ is a Markov kernel that lets the distribution $\mmu_k$ invariant; given $\{ \xi^1_k, \ldots, \xi^N_k \}$, the particles $\tilde{\xi}^1_k, \ldots, \wtilde{\xi}^N_k$ are independent. The subsequent particle approximation $\mmu^N_{k+1} = (1/N) \, \sum_{i=1}^N \bm{\delta}_{\xi_{k+1}^i}$ is obtained through a multinomial resampling step; the particles $\{ \xi^1_{k+1}, \ldots, \xi^N_{k+1} \}$ are $N$ i.i.d samples from the $\{ \xi^1_{k}, \ldots, \xi^N_{k} \}$-valued random variable that equals $\xi^i_{k}$ with probability $g_{k,k+1}(\xi^i_{k}) / [g_{k,k+1}(\xi^1_{k}) + \ldots + g_{k,k+1}(\xi^N_{k})]$. This procedures can be iterated to produce a sequence of particle approximations $\mmu^N_0, \ldots \mmu^N_n$. %
The output of the SMC algorithm employing $N \geq 1$ particles is an empirical approximation $\mmu^N_n$ to $\mmu_n = \mmu$,
\begin{align*}
\mmu^N_n \; = \; (1/N) \, \sum_{i=1}^N \bm{\delta}_{\xi^i_n}.
\end{align*}
The mutation and resampling steps are also frequently used in genetic optimization algorithms, and some of these algorithms can be analysed in terms of the same Feynman-Kac formulation, we refer the reader to \cite{DelMoralFeynmanKac} for more details. Asymptotic properties of the SMC algorithm are by now well understood (see e.g. \cite{moralmiclo}, and \cite{del2013mean} for a comprehensive overview). For ease of presentation, we will often present our results for functions with mean zero and finite moment of order $(1+\epsilon)$ for some $\epsilon > 0$; in other words, for a probability distribution $\ppi$, we consider the linear subspace
\begin{align*}
L^{2+}_0(\ppi) \defby 
\curBK{\phi: E \to \R \textrm{ such that }\ppi\BK{|\phi|^{2 + \epsilon}}<\infty 
\textrm{ for some }
\epsilon>0
\textrm{ and }
\ppi(\phi)=0 
}.
\end{align*}
%The central limit theorem was shown to hold for this class 
%of functions, in particular, 
Theorem $1$ of \cite{ChopinCLT} implies that for a test function $\phi \in L^{2+}_0(\mmu)$, the following limit holds in distribution,
\begin{align} \label{eq.CLT.statement}
\lim_{N \to \infty} \; 
N^{1/2} \, \sqBK{ \mmu^N_n-\mmu}(\phi)
\; = \; 
\Normal{0, V_n(\phi)},
\end{align}
with asymptotic variance $V_n(\phi)$ that can be expressed as
\begin{equation} \label{eq.asymp.var.expansion}
V_n(\phi) \;= \; 
\sum_{k=0}^n V_{k,n}(\phi)
\, \textrm{with} \,
V_{k,n}(\phi) \defby \norm{ \GG_{k,k+1}  \KK_{k+1}  \ldots \GG_{n-1,n}  \KK_n \, \phi }^2_{L^2(\mmu_k)}
\end{equation}
and $V_{n,n}(\phi) \defby \norm{ f }^2_{L^2(\mmu_k)} = \Var_{\mmu}(\phi)$. We note that the CLT and the expression \eqref{eq.asymp.var.expansion} was first proven for multivariate processes and more general Feynman-Kac models in the case of bounded functions in \cite{moralmiclo}. 
In the next section, we establish bounds on the asymptotic variance $V_n(\phi)$ under various natural conditions.

\section{Bounds on the asymptotic variance}
\label{SecResults}
In this section, we state and prove new asymptotic variance bounds for SMC empirical averages. Section \ref{SecUnimodal} considers bounds under global mixing assumptions. Section \ref{SecNoMixing} considers the multimodal case, under the assumption that there is no mixing between the modes. Finally, in Section \ref{SecWithMixing}, we obtain general results for multimodal distributions. To lighten the notations, for a positive operator $\bm{M}: L^2(\mmu) \to L^2(\mmu)$ and test functions $f,g \in L^2(\mmu)$, we set $\bra{f, g}_{\mmu,\bm{M}} \defby \bra{f, \bm{M} \, g}_{\mu}$ and $\norm{ \phi }^2_{L^2(\mmu),\bm{M}} \defby \bra{\phi, \bm{M} \, \phi}_{\mmu}$. In particular, we have that $\norm{\phi}_{L^2(\mmu_k), \GG_{k,k+1}} = \norm{\phi}_{L^2(\mmu_{k+1})}$, with $\GG_{k,k+1}$ defined as in Equation \eqref{eq.operator.G}.

\subsection{Bound under global mixing assumptions}
\label{SecUnimodal}
%It is known that $\nu_{n}^{N}(\phi)$ satisfies a CLT, with asymptotic %variance given by \eqref{eqVndef}. 
%We assume in this section that the Markov kernel $K_k$ is reversible with respect to $\mmu_k$.
The following theorem bounds the asymptotic variance $V_n(\phi)$ in terms of the ``global'' mixing properties of the Markov kernels $\KK_k$ and the size of the relative density $g_{k,k+1}$. 
%We denote the maximum of this, called the maximal density %ratio, as $\Gamma\defby\max_{0\le k\le n-1}\sup_{x\in \E}g_{k,k+1}$. 
Before stating our result, we need to introduce some notations.
Recall that $\mmu_k: L^2(\mmu_k) \mapsto L^2(\mmu_k)$ denotes the orthogonal  projection operator that maps a function $\phi$ to the  constant function that equals $\mmu_k(\phi)$ everywhere and that the Markov operator $\KK_k$ lets $\mmu_k$ invariant. We define the quantity 
\begin{align} \label{eq.gamma_K}
\gamma_{\KK} \defby 1 - 
\max \curBK{ \; \vertiii{ \KK_k - \mmu_k }_{L^2(\mmu_k)} 
\; : \; 
1 \leq k \leq n };
\end{align}
For any test function $\phi \in L^2(\mmu_k)$, we thus have that $\norm{(\KK_k - \mmu_k) \phi}_{L^2(\mmu_k)} \leq (1 - \gamma_{\KK}) \, \norm{\phi}_{L^2(\mmu_k)}$.
In the case where the Markov kernels $\KK_k$ are reversible, the quantity $\gamma_{\KK}$ is a uniform lower bound on their absolute spectral gap. The larger $\gamma_{\KK}$, the better the mixing properties of these Markov kernels.
%
%  MAIN THEOREM of the SECTION
%
\begin{theorem}[Variance bound under a global mixing assumption]\label{thmuni}
Let $\phi \in L^{2+}_0(\mmu)$ be a test function. Assume that
\begin{align} \label{eq.Gamma.gamma.unimodal}
\Gamma_g \; < \; \frac{1}{(1-\gamma_{\KK})^2}.
\end{align}
The CLT \eqref{eq.CLT.statement} holds with asymptotic variance $V_n(\phi)$ such that
\begin{align*}
V_n(\phi) 
\; \leq
\frac{1}{1-(1-\gamma_{\KK})^2 \cdot \Gamma_g} \, \Var_{\mmu_n}(\phi).
\end{align*}
\end{theorem}
%
%  PROOF
%
\begin{proof}
Recall the formula \eqref{eq.asymp.var.expansion} for the asymptotic variance.
First note that $V_{n,n}(\phi)=\Var_{\mmu_n}(\phi)$. By definition of the upper bound $\Gamma_{g}$ and the operators $\GG_{k,k+1}$, it follows that
\begin{align*}
V_{k,n}(\phi) 
&=
\norm{  \GG_{k,k+1}\KK_{k+1}\GG_{k+1,k+2}\KK_{k+2}\cdot\ldots \cdot\GG_{n-1,n}\KK_n \phi}_{L^2(\mmu_k)}^2\\
&\leq \Gamma_g \, 
\norm{  \KK_{k+1}\GG_{k+1,k+2}\KK_{k+2}\cdot \ldots \cdot\GG_{n-1,n}\KK_n \phi}_{L^2(\mmu_k), \GG_{k,k+1}}^2\\
&=
\Gamma_g \, \norm{  \KK_{k+1}\GG_{k+1,k+2}\KK_{k+2}\cdot\ldots \cdot\GG_{n-1,n}\KK_n \phi}_{L^2(\mmu_{k+1})}^2.
\end{align*}
Also, since the Markov kernel $\KK_j$ let $\mmu_j$ invariant and $\mmu(\phi) = 0$, we have that
\begin{align*}
\mmu_{k+1} \GG_{k+1,k+2}\KK_{k+2}\cdot\ldots \cdot\GG_{n-1,n}\KK_n \, \phi(x)  = 0
\end{align*}
for any $x \in E$.
Consequently, the quantity $\norm{  \KK_{k+1}\GG_{k+1,k+2}\KK_{k+2}\cdot\ldots \cdot\GG_{n-1,n}\KK_n \phi}_{L^2(\mmu_{k+1})}^2$ can also be expressed as 
\begin{align} \label{eq.thm.intermediat.1}
\norm{  \BK{\KK_{k+1} - \mmu_{k+1} }\GG_{k+1,k+2}\KK_{k+2} \cdot\ldots \cdot\GG_{n-1,n}\KK_n \phi}_{L^2(\mmu_{k+1})}^2.
\end{align}
The definition \eqref{eq.gamma_K} of $\gamma_{\KK}$ further yields that \eqref{eq.thm.intermediat.1} is less than
\begin{align}
\BK{1 - \gamma_{\KK}}^2 \, \norm{ \GG_{k+1,k+2}\KK_{k+2} \cdot\ldots \cdot\GG_{n-1,n}\KK_n \phi}_{L^2(\mmu_{k+1})}^2
\end{align}
so that $V_{k,n}(\phi) \leq \Gamma_g \, \BK{1 - \gamma_{\KK}}^2 \, \norm{ \GG_{k+1,k+2}\KK_{k+2} \cdot\ldots \cdot\GG_{n-1,n}\KK_n \phi}_{L^2(\mmu_{k+1})}^2$.
Keeping in mind that $\norm{\phi}_{L^2(\mmu_k)}^2 = \Var_{\mmu_k}(\phi)$, iterating the same arguments shows that
\begin{align*}
V_{k,n}(\phi) 
\; \leq \; 
\BK{ \Gamma_g \, \BK{1 - \gamma_{\KK}}^2 }^{n-k} \, \Var_{\mmu}(\phi).
\end{align*}
Since $V_n(\phi) = \sum_{k=0}^{n} V_{k,n}(\phi)$ and $\Gamma_g \, (1-\gamma_{\KK})^{2} < 1$, the conclusion follows.
\end{proof}
Theorem \ref{thmuni} gives an improvement over the quadratic error bounds provided by Theorem $1.2$ of \cite{schweizer2012non}; indeed, contrarily to their result, our bound on the asymptotic variance does not depend on the number $n \geq 1$ of resampling stages. Note that the required assumption \eqref{eq.Gamma.gamma.unimodal} can easily be enforced by including a sufficient number $n \geq 1$ of resampling stages and/or by increasing the amount of MCMC steps at  each stage. It is important to note that Theorem \ref{thmuni} does not assume that the target distribution $\mmu$ is unimodal in any sense; instead, assumptions on the global mixing properties of the Markov kernels $\KK_k$ are leveraged. However, and as is widely acknowledged in the Markov Chain Monte-Carlo literature, it is generally difficult to design Markov kernels with good global mixing properties for multimodal distributions. This is remark is one of main motivations for our work; in the next sections, we describe results that do not require the Markov kernels $\KK_k$ to possess good global mixing properties.

\subsection{Bound for the multimodal case}\label{SecWithMixing}
In this section we examine the case of multiple modes.  Mixing between modes is allowed. We look at partitions of the state space $E$ that may vary with the algorithm index $k$,
\begin{align*}
E \defby \bigsqcup_{j=1}^{m(k)} \; F^{(j)}_k.
\end{align*}
This means that one allows different modes, and potentially a different number of modes, for each intermediate distribution $\mmu_k$. This extra generality allows us to analyse a wider range of interpolating distributions. In particular, we will use this property in the analysis of the Potts model (see Section \ref{SecPotts}).
%Moreover, we do not assume in this section that the Markov kernels $\curBK{ \KK_k }_{k=1}^n$ are reversible. 
We define the  \emph{growth-within-mode constant} as
\begin{align} \label{growthwithinmodedefeq}
B_{k,k+1}\defby \max \curBK{ \mmu_{k+1}\left(F^{(r)}_k\right) \, / \, \mmu_{k}\left(F^{(r)}_k\right) \; : \; 1\leq r \leq m(k) }.
\end{align}
The restriction of $\mmu_k$ to $F^{(r)}_k$, denoted by $\mmu_{k,r}$, is defined as
\[\mmu_{k,r}(S)\defby\frac{\mmu_k\left(S\cap F^{(r)}_k\right)}{\mmu_{k}(F^{(r)}_k)} \text{ for every measurable  }S\subset E.\]
Consider the situation where, as is common in practice, the Markov kernel $\KK_k$ is of the form $\KK_k=\bm{P}_k^{t_k}$ for some Markov kernel $\bm{P}_k$; in words, the kernel $\KK_k$ corresponds to iterating $t_k$ steps of the Markov kernel $\bm{P}_k$. We introduce an approximation called \emph{metastable state}, which is a kernel $\widehat{\mmu}_k$ defined as

\begin{align} \label{eq.pi.hat}
\widehat{\mmu}_k(x, \phi)
\; = \;
\sum_{ r = 1 }^{m(k)} \;
\alpha_{k,r}(x) \, \mmu_{k,r}\BK{ \phi_{|F^{(r)}_k} }
\end{align}
where for every $x \in E$ and index $1 \leq k \leq n$ the family $\{\alpha_{k,r}(x)\}_{r=1}^{m(k)}$ is a sequence of non-negative real numbers that are such that
\begin{align} \label{eq.alpha.condition}
\sum_{r=1}^{m(k)} \; \alpha_{k,r}(x)
\; \leq \; 1 
\end{align}
Since $\sum_{r=1}^{m(k)} \alpha_{k,r}(x)$ can be strictly smaller than one, the metastable operator $\widehat{\mmu}_k$ is not necessarily a Markov kernel. 
A natural choice is $\alpha_{k,r}(x) = \KK_k(x,F^{(r)}_k )$ (the probability of ending up in mode $F^{(r)}_k$ when started from $x$). Another possibility, useful when the chain mixes well globally, consists in setting $\alpha_{k,r}(x) = \mmu_k( F^{(r)}_k )$; this approximation results in $\widehat{\mmu}_k(x, dy) = \mmu_k(dy)$.
As will become clear in Section \ref{sec.metastable.approx}, for a suitable choice of coefficients $\alpha_{k,r}(x)$, the approximation $\KK_k\approx \widehat{\mmu}_k$ is often accurate, even for reasonably small values of $t_k$. The following result is our variance bound in this setting.

%
%  MAIN THEOREM
%
\begin{theorem}[Variance bound for multimodal case with mixing]\label{thmwithmixing1}
Assume that $\Gamma_g < \infty$. For a bounded and measurable test function $\phi$, the CLT \eqref{eq.CLT.statement} holds with an asymptotic variance $V_n(\phi) = \sum_{k=0}^{n} V_{k,n}(\phi)$ where, for any index $0 \leq k \leq n-1$, we have
\begin{align}\label{eqVjnwithmixing}
V_{k,n}(\phi)
&\leq
\Gamma_g \,
\prod_{j=k+1}^{n-1}
\curBK{ B_{j,j+1}  + \Gamma_g \, \norm{\KK_{j} - \widehat{\mmu}_j}_\infty } \, \norm{\phi}^2_\infty.
\end{align}
\end{theorem}

%
%  PROOF STARTS HERE
%
\begin{proof}
%It is technically simpler to work with the $L_{\infty}$ norm instead %of the $L^2(\mmu_k)$ norm. 
%As in the proof of Theorem \ref{thmnomixing}, 
Since $\|\GG_{k,k+1}\|_{\infty}\le \Gamma_g$, we have
\[V_{k,n}(\phi) \leq \Gamma_g \, \norm{ \KK_{k+1} \, \GG_{k+1,k+2} \cdot\ldots \cdot\GG_{n-1,n}\KK_n \phi}^2_{L^2(\mmu_{k+1})}.\]
Moreover, $\norm{ \KK_{k+1} \, \GG_{k+1,k+2} \cdot\ldots \cdot\GG_{n-1,n}\KK_n \phi}^2_{L^2(\mmu_{k+1})}$ is less than
\begin{align*}
\norm{ \KK_{k+1} \, \GG_{k+1,k+2} \cdot\ldots \cdot\GG_{n-1,n}\KK_n \phi}_{\infty} \times \norm{\KK_{k+1} \, \GG_{k+1,k+2} \cdot\ldots \cdot\GG_{n-1,n}\KK_n \, \phi}_{L^1(\mmu_{k+1})}.
\end{align*}
Since the Markov kernel $\KK_{k+1}$ lets $\mmu_{k+1}$ invariant, this is a contraction in $L^1(\mmu_{k+1})$; consequently
\begin{align*}
\, & \norm{\KK_{k+1} \, \GG_{k+1,k+2} \, \KK_{k+2} \cdot\ldots \cdot\GG_{n-1,n}\KK_n \, \phi}_{L^1(\mmu_{k+1})}
\\
&\leq 
\norm{\GG_{k+1,k+2} \, \KK_{k+2}\cdot\ldots \cdot\GG_{n-1,n}\KK_n \, \phi}_{L^1(\mmu_{k+1})}\\
&=
\norm{\KK_{k+2} \cdot\ldots \cdot\GG_{n-1,n}\KK_n \, \phi}_{L^1(\mmu_{k+2})}
\leq \ldots \leq 
\norm{\phi}_{L^1(\mmu)} \leq \norm{\phi}_{\infty}.
\end{align*}
Also, since $\norm{\KK_n \, \phi}_{\infty} \leq \norm{\phi}_{\infty}$, we have that
\begin{align*}
\norm{ \KK_{k+1} \, \GG_{k+1,k+2} \cdot\ldots \cdot\GG_{n-1,n}\KK_n \phi}_{\infty}
\leq
\curBK{ \prod_{j=k+1}^{n-1}
\vertiii{\KK_j \, \GG_{j+1}}_{\infty} } \, 
\norm{\phi}_{\infty}.
\end{align*}
Furthermore, 
$\vertiii{\KK_j \, \GG_{j+1}}_{\infty} 
\leq 
\vertiii{\widehat{\mmu}_j\, \GG_{j+1}}_{\infty} 
+ \vertiii{\BK{ \KK_j - \widehat{\mmu}_j} \, \GG_{j+1}}_{\infty}$. Definition \eqref{growthwithinmodedefeq} yields that $\vertiii{\widehat{\mmu}_j\, \GG_{j+1}}_{\infty}$ is less than $B_{j,j+1}$; similarly, $\vertiii{\BK{ \KK_j - \widehat{\mmu}_j} \, \GG_{j+1}}_{\infty}\le\Gamma_g \, \vertiii{\KK_j - \widehat{\mmu}_j }_{\infty}$. It follows that
\begin{align*}
\vertiii{\KK_j \, \GG_{j+1}}_{\infty}
\; \leq  \; 
B_{j,j+1} + \Gamma_g \, \vertiii{\KK_j - \widehat{\mmu}_j }_{\infty},
\end{align*}
as required.
\end{proof}
Note that unlike Theorem \ref{thmuni}, here we use the supremum norm (thus our result is restricted to bounded functions), because we have encountered some technical difficulties when using $\|\cdot \|_{L^2(\mmu_{k})}$ norms in this setting. The main improvement in this theorem over the results of \cite{schweizer2012non} is that mixing is allowed between the modes. For completeness, a variant of the asymptotic variance bound of \cite{schweizer2012non} (when no mixing is allowed between the modes) is presented in Section \ref{SecNoMixing} of the Appendix.

\subsection{Framework for metastable approximation}
\label{sec.metastable.approx}
In order to apply Theorem \ref{thmwithmixing1}, one needs to bound the norm $\norm{ \KK_{k}-\widehat{\ppi}_{k} }_{L^\infty}$ for $1\le k\le n$. This section provides with a framework for establishing such bounds.
Suppose that the state space $E$ is partitioned into modes $\{ \F^{(j)} \}_{j=1}^m$ and that each mode is comprised of an \emph{inner region} $\I^{(j)}$ and a \emph{border region} $\B^{(j)}$; in other words, we have the following decomposition of the state space,
\begin{align*}
E \; = \; \bigsqcup_{i=1}^m \F^{(j)}
\; = \; \bigsqcup_{i=1}^m \curBK{ \I^{(j)} \sqcup \B^{(j)} }.
%\qquad \textrm{with} \qquad
%\F^{(j)} = \I^{(j)} \sqcup \B^{(j)}.
\end{align*}
It will reveal useful to set
\begin{align*}
\I = \bigsqcup_{j=1}^m \I^{(j)}
\qquad \textrm{and} \qquad
\B = \bigsqcup_{j=1}^m \B^{(j)}.
\end{align*}
For every $1 \leq j \leq m$, we denote restrictions of $\mmu$ to $\F^{(j)}$ by $\mmu^{(j)}$, defined by the relation
\begin{equation}\label{murestricteddefeq}
\mmu^{(j)}(\phi) = \frac{\mmu(\phi_{| \F^{(j)}})}{\mmu(\F^{(j)})}.
\end{equation}
For $x \in E$ and an integer $t \geq 1$, consider the quantity
\begin{align*}
q^{(i)}(x,t)
\defby
\PP\BK{ \left. X_{\tau_\B} \in \I^{(i)}, \tau_\B \leq t \right| X_0=x},
\end{align*}
where $\{X_k\}_{k \geq 0}$ is a Markov chain with Markov transition kernel $\P$, and $\tau_\B$ is the time of exit from $\B$,  $\tau_\B \defby \inf \curBK{t \geq 0: X_t \not \in \B}$. This expresses the probability that we exit the border regions in one of the first $t$ steps, and the first step outside $\B$ is in the inner region $\I^{(i)}$ . Our main result in this section, Theorem \ref{thmwithmixing2} quantifies the approximation $\P^t \approx \widehat{\ppi}^{(t)}$, where the kernel $\widehat{\pi}^{(t)}$ is defined as
\begin{align}
\label{widehatpitdefeq}\widehat{\pi}^{(t)}(x,dy)
\; \defby \; 
\left\{
\begin{array}{ll}
\mmu^{(j)}(dy)  &  \text{ for } x \in \I^{(j)} \\
\sum_{j} q^{(j)}(x,\lfloor t/2\rfloor) \cdot \mmu^{(j)}(dy) & \text{ for } x \in \B.
\end{array}   
\right.
\end{align}
Note that $\sum_{i=1}^{m(k)}q^{(i)}(x,t)=\PP(\tau_\B \leq t)\le 1$ so condition \eqref{eq.alpha.condition} is satisfied. The bound on the discrepancy $\vertiii{ \mtx{P}^t-\widehat{\ppi}^{(t)} }_{\infty}$ is expressed in terms of the event $\SB(x,t)$ (stay in the border region) defined as
\begin{align*}
\SB(x,t)
&\defby
\big\{
\text{start at $x\in \B$ and stay inside $\B$ for $t$ steps}
\big\}.
\end{align*}
When starting in the inner region of a mode and after a number of steps slightly larger than the local mixing time, the Markov chain is typically approximately distributed according to the restriction of the stationary distribution to that mode; nevertheless, the Markov chain is still not likely to escape from that mode. When starting from a border region, the Markov chain typically enters the inner region rapidly then stay there in the rest of the steps, and mix well within that mode. The number of steps thus needs to be chosen carefully to make sure that we exit the border regions and mix well within the modes, but do not exit the modes once we have entered an inner region (due to the ``potential well'' effect).

%
% MAIN THEOREM
%
\begin{theorem}[Quantifying the quality of metastable approximation]\label{thmwithmixing2}
Let $\widehat{\ppi}^{(t)}$ be defined as in \eqref{widehatpitdefeq}, then the following bound holds,
\begin{equation} \label{Ptpitbound}
\begin{aligned}
\vertiii{ \P^t - \widehat{\ppi}^{(t)} }_{\infty}
&\leq 
\max_{x \in \B} \; \PP \BK{ \SB(x,\lfloor t/2\rfloor) } \\
&+2 \, \max_{1\leq i \leq m} \, \max_{\lceil t/2\rceil \leq r \leq t}\, \sup_{x \in \I^{(i)}} \; \dtv \BK{ \P^{r}(x,\cdot),\mmu^{(i)}}.
\end{aligned}
\end{equation}
\end{theorem}
%
%
%
% PROOF STARTS HERE
%
\begin{proof}
For two distributions $\eta_1, \eta_2$ on $E$ (which are not necessarily probability distributions), we define their total variational distance as 
\[\dtv(\eta_1,\eta_2)\defby\frac{1}{2}\sup_{f:E\to [-1,1]} |\eta_1(f)-\eta_2(f)|,\]
where the supremum is taken among Borel-measurable functions from $\E$ to $[-1,1]$.
By the definition of $\vertiii{ \P^t-\widehat{\ppi}^{(t)} }_{\infty}$, we can rewrite it as
\[\vertiii{ \P^t-\widehat{\ppi}^{(t)} }_{\infty}=2\sup_{x\in E} \dtv\left(\P^t(x,\cdot),\widehat{\ppi}^{(t)}(x,\cdot) \right),\]
so we need to bound this total variational distance for every $x\in E$. We consider two separate cases.

Fist, assume that $X_0 = x \in \I^{(i)}$ for some $1\leq i \leq m$. In this case, $q^{(i)}(x,\lfloor t/2\rfloor)=1$, so $\widehat{\ppi}^{(t)}(x,\cdot)=\mmu^{(i)}(\cdot)$, and thus
\[2\dtv(\P^t(x,\cdot),\widehat{\ppi}^{(t)}(x,\cdot))=2\dtv\left(\P^t(x,\cdot), \mmu^{(i)}\right),\]
which is bounded by the right hand side of \eqref{Ptpitbound}.

Alternatively, assume that $X_0 = x \in \B$. 
For $1\le i\le m$, define the events 
\[E^{(i)}\defby\left\{X_0=x, X_{\tau_\B} \in \I^{(i)}, \tau_\B \leq \lfloor t/2\rfloor \right\}.\] Let $E^{\cup}\defby\cup_{1\le i\le m}E^{(i)}$, and $E^{c}=(E^{\cup})^c$ (the complement of $E^{\cup}$). Then $\PP(E^{(i)})=q^{(i)}(x,\lfloor t/2\rfloor)$, and we can write the two kernels $\P^t(x,dy)$ and $\widehat{\ppi}^{(t)} (x,dy)$ as
\begin{align*}
\P^{t}(x,dy)&=\PP(X_t\in dy|E^{c})\cdot \PP(E^{c})+\sum_{i=1}^{m} \PP(X_t\in dy|E^{(i)})\cdot \PP(E^{(i)}),\\
\widehat{\ppi}^{(t)} (x,dy)&=\sum_{i=1}^{m}\mmu^{(i)}(dy)\cdot \PP(E^{(i)}).
\end{align*}
Based on this, by the triangle inequality, we have
\begin{align*}
%\label{l1distptpiteq1}
2\dtv(\P^t(x,\cdot),\widehat{\ppi}^{(t)}(x,\cdot))\le
\PP(E^c)+2\sum_{i=1}^{m} \PP(E^{(i)})\cdot \dtv\left(\LL(X_t|E^{(i)}), \mmu^{(i)}\right),
\end{align*}
where $\LL(X_t|E^{(i)})$ denotes the law of $X_t$ conditioned on $E^{(i)}$. The result now follows from the fact that $\PP(E^c)\le \max_{x \in \B} \; \PP \BK{ \SB(x,\lfloor t/2\rfloor) }$ and 
\[\dtv\left(\LL(X_t|E^{(i)}), \mmu^{(i)}\right)\le \max_{1\leq i \leq m} \, \max_{\lceil t/2\rceil \leq r \leq t}\, \sup_{x \in \I^{(i)}} \; \dtv \BK{ \P^{r}(x,\cdot),\mmu^{(i)}}.\qedhere\]
\end{proof}
To use Theorem \ref{thmwithmixing2}, one does not need to know the exact probabilities $\P^t(x,\F^{(i)})$ of ending up in the different modes. Instead, one only needs to estimate the probability of  staying in the border regions for many steps, and the total variational distance of $\P^t(x,\cdot)$ from the local restriction $\mmu^{(i)}$ when started from a point $x$ in an inner region $\I^{(i)}$. As shall be seen in the application, these quantities can typically be bounded using concentration inequalities and drift arguments. 
%Finally, note that, although our approach is rather straightforward, the study of the metastability properties of Markov processes is an active area of research and sophisticated methods are available \citep{BovierMetastability}.

\section{Interpolation to independence sequence}\label{secinterpolationindependence}
Suppose that we have a sequence of probability distributions $(\eta_k)_{k\in \Z_+}$ defined on $(\Lambda^k)_{k\in \Z_+}$, respectively (which are  product spaces of increasing dimension), and that our target is $\eta_d$. Suppose that these distributions satisfy some sort of \emph{scale invariance property}. By this, we mean that for sufficiently high $k$, the number, position, and probability mass of the modes are essentially constant in some appropriate coordinate system (i.e. the positions of the modes have approximately reached a limit).
For such systems, we define the  \emph{interpolation to independence sequence} as a sequence of distributions on $\Lambda^d$, denoted by $\mmu_0, \ldots, \mmu_d$ such that $\mmu_d:=\eta_d$ (i.e. the target measure in dimension $d$), and $\mmu_k$ corresponds to the distribution when the first $k$ coordinates are distributed according to $\eta_k$, and the rest of the coordinates are i.i.d. uniformly distributed on $\Lambda$ (independently of the first $k$ coordinates). For the SMC sampler based on this sequence, we use some appropriate MCMC kernels $\mtx{K}_k$ that first change the first $k$ coordinates (such as versions of the Glauber dynamics), and then replace the rest of the coordinates by independent copies.

The interpolation to independence sequence consists of miniaturised versions of the original system (and some independent coordinates to keep the state space invariant). As we are going to see, if the system satisfies the scale invariance property, then the number and location of the modes is essentially the same across all the distributions. This ensures that the growth-within-mode constants $B_{k,k+1}$ are small, and thus the method is efficient if the MCMC moves are chosen appropriately (see Theorems \ref{thmwithmixing1} and \ref{thmwithmixing2}). This is a key difference with the standard tempering sequence, because the change of the temperature parameter might alter the number and location of the modes drastically, so the growth-within-mode constants might be very large (see \cite{BhatnagarRandall}).

%The MCMC moves that we are going to use for $\mmu_k$ do one step in the Glauber dynamics for $(\sigma_1,\ldots, \sigma_k)$, and replace the rest of the spins by independent copies. We denote the Markov kernel corresponding to this move by $\P_k$, and the kernel combining $t_k$ such steps by $\KK_k=\P_k^{t_k}$. Then it is easy to see that this kernel is a reversible with respect to $\mmu_k$.

%A key difference between this new sequence of distributions and the standard tempering sequence is that now the inverse temperature parameter $\tbeta$ is the same in every distribution in the sequence. This means that the sequence consists of miniaturised versions of the original system of spins (and some independent spins to keep the state space invariant).
%As we are going to see, the general behaviour of this system does not depend on the number of spins, and thus the number and location of the modes is essentially the same across all the distributions. This ensures that the growth-within-mode constants $B_{k,k+1}$ are small, thus we are able to avoid the pitfalls encountered when using the standard tempering sequence.

The idea of interpolating to independence have appeared in the literature of Stein's method, see Section 3.4 of \cite{Steincouplings}. At this point, we note that somewhat similar ideas have appeared for SMC methods in \cite{multilevelsmc}, where a gradual coarsening of the grid is used for the solution of PDEs, and in \cite{lindsten2014divide}, where graphical models are studied, and the interpolating sequence is chosen by breaking them into smaller blocks gradually. Multigrid methods have been fruitful for solving challenging problems in numerical analysis, and MCMC samplers based on this idea have been also proposed in \cite{goodman1989multigrid}. In addition, some variations to the standard tempering distributions in have been proposed  in the literature to parallel tempering, such as truncating the peaks (\cite{Equienergy}) and moving across different dimensional spaces (\cite{LiuChiara}). 

However, most of these works are lacking in theoretical explanation of the efficiency of the proposed methodology. Moreover, it is not clear to us whether they can overcome the type of problem we encounter for the Potts model (where some of the modes have exponentially small probability according the initial distribution, and thus they might take exponentially long time to be discovered). We note that Theorem 7.7 of \cite{bhatnagar2015simulated} has shown that for some model specific interpolating distributions called entropy dampening distributions, the simulated tempering MCMC sampler for the Potts model mixes in polynomial time.

The following section states our application in this paper, a bound on the asymptotic variance of the SMC method using the interpolation to independence sequence of distributions applied to the Potts model. We would like to emphasise the fact that the interpolation to independence methodology is not limited to this single example. A natural generalisation is to use it for sampling from exponential random graph models. For such models, MCMC sampling has been shown to mix very slowly, and the location of the modes depends of the temperature parameter (see \cite{ChatterjeeExpRandGraphs}). However, they are believed to be scale invariant, so our method could be useful for sampling from them.
\section{Application to the Potts model}\label{SecPotts}
\subsection{Introduction}
In this section, we are going to study the Potts model introduced in \cite{PottsRB}. Let $G$ be a simple graph with $M$ vertices. The model consists of $M$ spins $\sigma:=(\sigma_1,\ldots,\sigma_M)$ taking values in $\Omega:=\{1,\ldots,q\}^M$ for some $q\ge 2$ (these are called ``colours''). The Hamiltonian of the model is defined as
\begin{equation}
H(\sigma)=-\sum_{(i,j)\in G}J\cdot 1_{[\sigma_i=\sigma_j]},
\end{equation}
where the summation is over all the edges in $G$. The sign of $J$ determines whether the neighbors prefer the same colour (ferromagnetic case) or different colour (antiferromagnetic case).
The Gibbs distribution on configurations is given by
\begin{equation}
\mmu_{\tbeta,G}^{\mathrm{Potts}}(\sigma):=\frac{\exp(-\beta\cdot H(\sigma))}{Z(\beta)},
\end{equation}
where $\beta$ is the inverse temperature parameter (a constant independent of $\sigma$), and $Z(\beta)$ is the normalising constant.

We will consider the 3 colour mean-field case, when $G$ is the complete graph, and $q=3$. In this case, a simple rearrangement (see \cite{BhatnagarRandall}) shows that we can equivalently write the model as 
\begin{equation}
\mmu_{\tbeta,M}^{\mathrm{Potts}}(\sigma):=\frac{\exp\left(M \tbeta\left(s_1(\sigma)^2+s_2(\sigma)^2+s_2(\sigma)^2\right)\right)}{\tilde{Z}(\tbeta)},
\end{equation}
where $s_k(\sigma)$ is the ratio of spins of colour $k$  for $k=1, 2, 3$, $\tbeta:=\frac{\beta J}{2}M$, and $\tilde{Z}(\tbeta)$ is the new normalising constant. We call the triple $s_1(\sigma),s_2(\sigma),s_3(\sigma)$ the magnetisation vector.

This model is known to undergo a first order phase transition at $\tbeta_c=2\log(2)$. We denote the distribution of the magnetisation vector as
\[\mmu_{\tbeta,M}^{\mathrm{mag}}(s_1,s_2,s_3):=\mmu_{\tbeta,M}^{\mathrm{Potts}}(\{\sigma: s_1(\sigma)=s_1,s_2(\sigma)=s_2,s_3(\sigma)=s_3\}).\]
To show the difference between the different phases, we include 3 contour plots of 
\[\log\mmu_{\tbeta,M}^{\mathrm{mag}}(s_1,s_2,s_3)=\log\left(\binom{M}{M s_1}\binom{M-M s_1}{Ms_2} \cdot \frac{\exp\left(M \tbeta\left(s_1^2+s_2^2+s_3^2\right)\right)}{\tilde{Z}(\tbeta)}\right)\]
for $\tbeta=\tbeta_c/2$, $\tbeta=\tbeta_c$ and $\tbeta=2\tbeta_c$, for $M=1000$.
The plots show $s_1$, $s_2$ and $s_3$ in a barycentric coordinate system, the darker colours correspond to areas of higher probability.
\begin{figure}[ht]
  \centering
%  \subcaptionbox{$\tbeta<\tbeta_c$\label{figcontour1}}{\includegraphics[height=4cm]{potts_contour_1nn.eps}}\hspace{1em}
%  \subcaptionbox{$\tbeta=\tbeta_c$\label{figcontour2}}{\includegraphics[height=4cm]{potts_contour_2nn.eps}}\hspace{1em}
%  \subcaptionbox{$\tbeta>\tbeta_c$\label{figcontour3}}{\includegraphics[height=4cm]{potts_contour_3nn.eps}}
  \subcaptionbox{$\tbeta<\tbeta_c$\label{figcontour1}}{\includegraphics[height=3.4cm,bb=0 0 637 516]{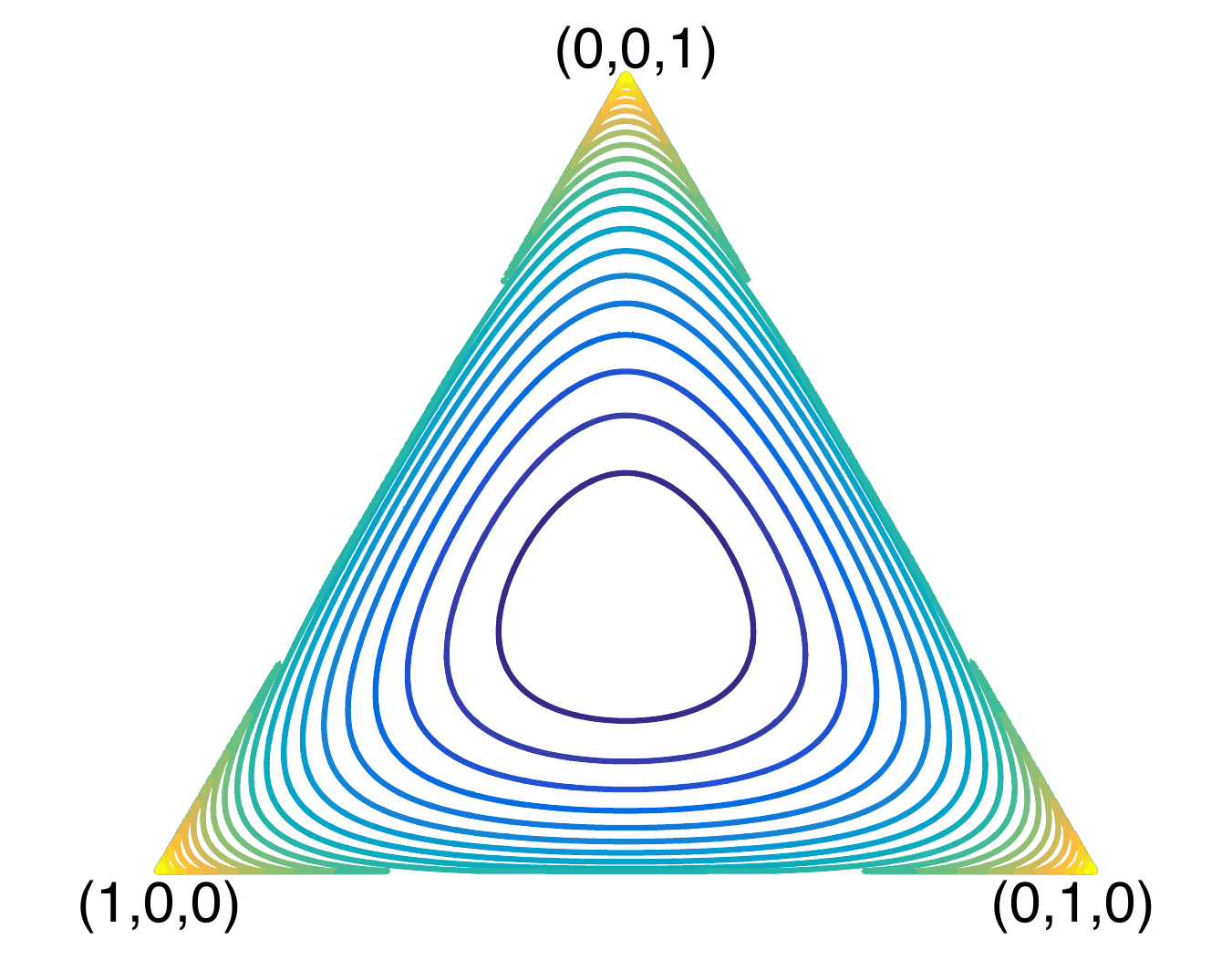}}\hspace{1em}
  \subcaptionbox{$\tbeta=\tbeta_c$\label{figcontour2}}{\includegraphics[height=3.4cm,bb=0 0 637 516]{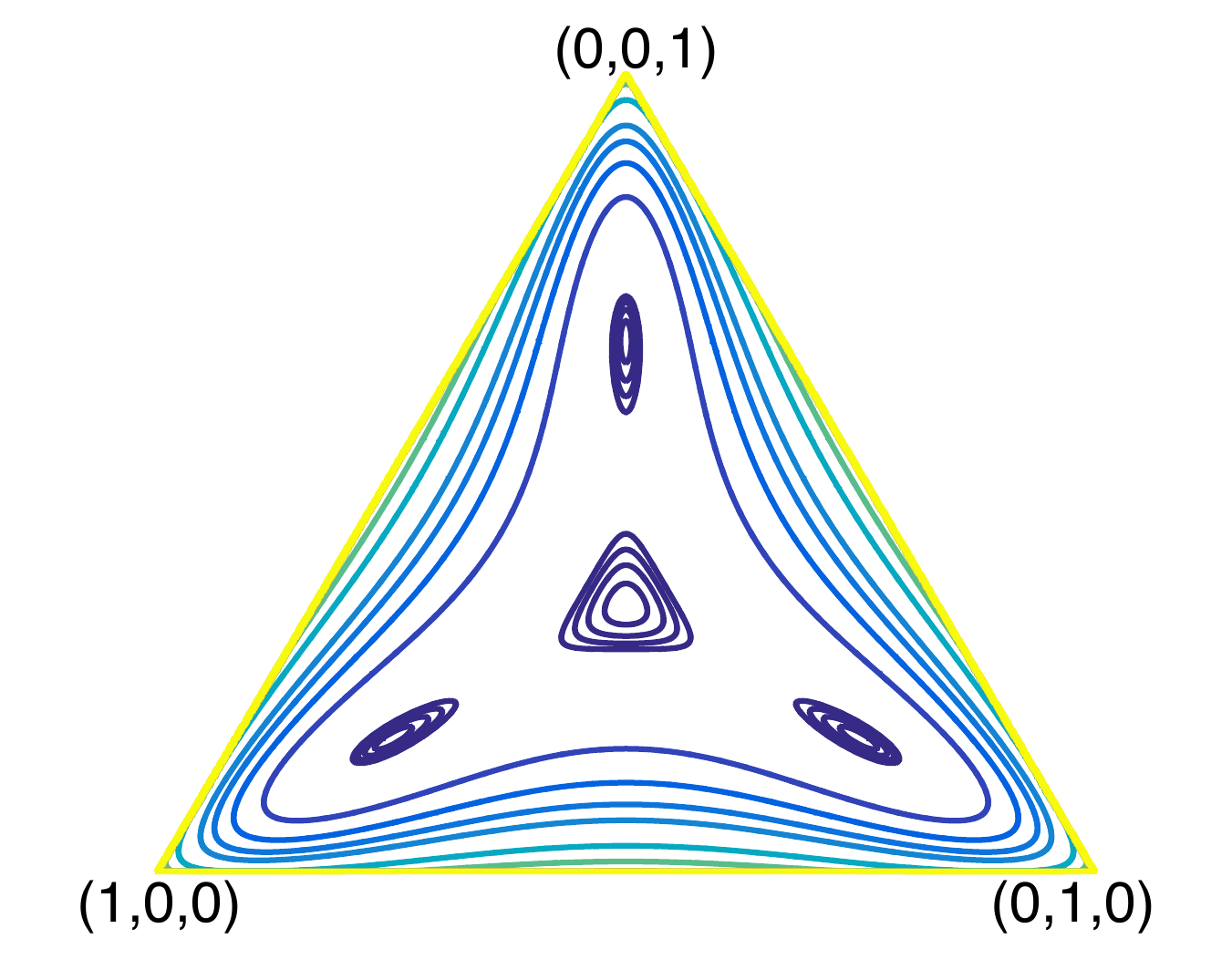}}\hspace{1em}
  \subcaptionbox{$\tbeta>\tbeta_c$\label{figcontour3}}{\includegraphics[height=3.4cm, bb=0 0 637 516]{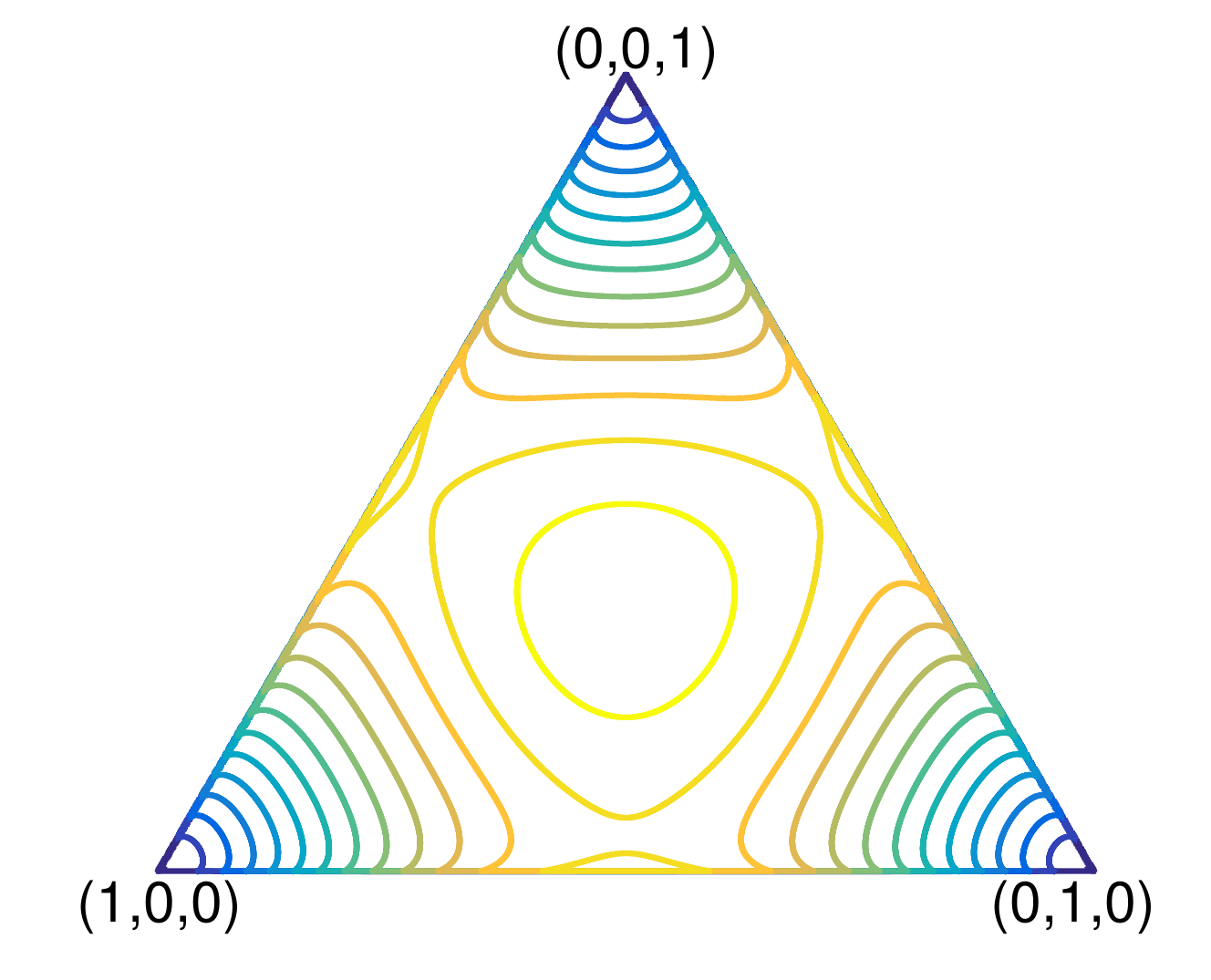}}

\caption{Contour plots of the log-likelihood as a function of $(s_1,s_2,s_3)$}
\end{figure}
As we can see, for $\tbeta<\tbeta_c$, there is a single local maximum centered at $(s_1, s_2, s_3)=\left(\frac{1}{3},\frac{1}{3},\frac{1}{3}\right)$. At $\tbeta=\tbeta_c$, there are 4 local maximums, centered at 
\begin{equation}\label{C1234defeq}
C_{1}:=\left(\frac{2}{3},\frac{1}{6},\frac{1}{6}\right), C_{2}:=\left(\frac{1}{6},\frac{2}{3}, \frac{1}{6}\right),  C_{3}:=\left(\frac{1}{6},\frac{1}{6},\frac{2}{3}\right),\hspace{0.5mm}\text{and}\hspace{1.1mm}C_{4}:=\left(\frac{1}{3},\frac{1}{3},\frac{1}{3}\right).\end{equation}
Finally, for $\tbeta>\tbeta_c$, there are 3 local maximums, centered at $(1,0,0)$, $(0,1,0)$ and $(0,0,1)$. 

The Glauber dynamics Markov chain updates a randomly chosen spin conditioned on the rest of the spins in each step. \cite{PeresPotts} has shown that this chain is fast mixing in part of the region $\tbeta<\tbeta_c$ (called the \emph{high temperature} region), but its mixing time increases exponentially in the number of spins for $\tbeta\ge \tbeta_c$. This phenomenon is caused by the existence of multiple modes for $\tbeta\ge \tbeta_c$. 

Parallel tempering (also called Metropolis coupled MCMC) is a popular method that has been shown to work well for some multimodal distributions (see \cite{MadrasZheng}, \cite{Woodardrapidmixing}). However, it was shown in \cite{BhatnagarRandall} that parallel and simulated tempering will have exponentially slow mixing time for $\tbeta\ge \tbeta_c$ if the tempering distributions with the standard temperature ladder are used. The reason for this is that the 3 new modes that appear for $\tbeta\ge \tbeta_c$ have very little probability according to the uniform distribution on all configurations. This means that even if theoretically we could move between the modes in the levels of high temperature, practically we will almost never move to the 3 new modes.

The goal of this section is to show that this difficulty can be overcome by SMC methods using the \emph{interpolation to independence} sequence of distributions defined in Section \ref{secinterpolationindependence}. For simplicity, we will only consider the case $\tbeta=\tbeta_c$ (but we believe that similar arguments work for any temperature and number of colours). The MCMC moves that we are going to use for $\mmu_k$ do one step in the Glauber dynamics for $(\sigma_1,\ldots, \sigma_k)$, and replace the rest of the spins by independent copies. We denote the Markov kernel corresponding to this move by $\P_k$, and the kernel combining $t_k$ such steps by $\KK_k=\P_k^{t_k}$. Then it is easy to see that this kernel is a reversible with respect to $\mmu_k$.

It is not difficult to see heuristically that if we applying the SMC algorithm with Glauber dynamics steps, and choose the interpolating distributions $\mmu_i$ as $\mmu_{\tbeta_c\cdot (i/n),M}^{\mathrm{Potts}}$ for $0\le i\le n$ (tempering distributions), then most of the particles would stay close to the central mode, and they would never discover the 3 other modes (we have some numerical evidence for this). For such a sequence of distributions, the product of the growth-within-mode constants, $\prod_{j=0}^{M-1}B_{j,j+1}$, grows exponentially with $M$.

\subsection{Main result}
\begin{theorem}[SMC variance bound for the Potts model]\label{thmvarPotts}Suppose that $\tbeta=\tbeta_c$. There is a constant $C_1\in \R_+$ such that for the SMC sampler described above, assuming that the number of MCMC steps in stage $k$ is chosen as $t_k=\lceil C_1 k \log(k)^2\rceil$,
the asymptotic variance of the SMC empirical average of any bounded function $f:\Omega\to \R$ satisfies that
\begin{equation}\label{eqPottsvarbound}V_M(f)\le  C_2 M \|f\|_{\infty}^2,\end{equation}
for some absolute constant $C_2\in \R_+$.
\end{theorem}
Due to space considerations we only prove this result for $\tbeta=\tbeta_c$, but we believe that a similar result holds for the Potts model with any number of colours, and any temperature parameter $\tbeta$. Here we note that the overall amount of computational effort needed to obtain a sample of unit variance by this algorithm is $\OO(M^3 \log^2(M))$. This is significantly better than the mixing rate obtained in \cite{bhatnagar2015simulated}. We think that this could be improved to a smaller power of $M$ by resampling only when the effective sample size parameter (ESS) is below a certain threshold (see \cite{DDJJRRSB}), since resampling is not necessary in each stage because the ratio of particles in the separate modes is converging quickly to a limit.

%We believe that the applicability of this method goes beyond this particular example, and it could be applied to other complex systems, where the interpolating sequence to the target distribution can be constructed in an evolutionary way. In fact, the Potts model and various spin glass models have found surprising applications to diverse areas, for an overview, we refer the reader to the monographs \cite{Nishimoribook} and \cite{stein2013spin}. The analysis of such examples is beyond the scope of this paper and it is left for future research.

\section{Proofs for the Potts model}\label{SecProofPotts}
The proof of Theorem \ref{thmvarPotts}, based on our theoretical results, is rather complex. 
To make the presentation clear, in Section \ref{seckeypropsandproof} we first state 5 key propositions bounding the maximal density ratio ($\Gamma_g$) and growth-within-mode ($B_{j,j+1}$) constants in Theorem \ref{thmwithmixing1} and the 3 terms in Theorem \ref{thmwithmixing2}  (the probability of escaping from the inner regions, the probability of staying in the border regions, and the total variational distance to the local restriction of the stationary distribution to the mode when started from a place in one of the inner regions). Based on these propositions, we then prove Theorem \ref{thmvarPotts}.

In Section \ref{secpreliminaryPotts}, we show some preliminary results that will be used in the proof of the key propositions. This is followed by Section \ref{secdrift}, where we show a lemma bounding the drift of the Glauber dynamics chain. Finally, in Sections \ref{secproofprop1} - \ref{secproofprop3}, we prove the key propositions used in the proof of Theorem \ref{thmvarPotts}.

In the rest of this paragraph, we introduce some notations that will be used through the proof.
Let $\sigma(0)\in \Omega=\{1,2,3\}^M$ be a fixed starting point, and $\sigma(0),\sigma(1), \sigma(2),\ldots$ be a realisation of the Glauber dynamics chain started at $\sigma(0)$. Let 
\[S(k):=(s_1(\sigma(k)),s_2(\sigma(k)),s_3(\sigma(k)))\]
be the vector of the ratios of different colours in $\sigma(k)$. It is shown in \cite{PeresPotts} that $S(k)$ is also a Markov chain, called the \emph{magnetisation chain}. We call the state space of this chain $\Omega^{S}$. In order to understand the geometry of $\Omega^{S}$, we will think of each point of it in barycentric coordinates.  We will call by $T:=\{(s_1,s_2,s_3): 0\le s_1,s_2,s_3\le 1, s_1+s_2+s_3=1\}$ the \emph{main triangle}. On Figure \ref{fig4modes}, we illustrate this triangle with an equilateral triangle with side length 1. The centroid of this triangle is point $C_4$, and we denote the 3 vectors pointing from $C_4$ to 
the three corners of an equilateral triangle by $e_1$, $e_2$ and $e_3$ (each of them has length $\frac{\sqrt{3}}{3}$). Then to each point $s=(s_1,s_2,s_3)\in T$, the sum $C(s):=s_1 e_1+s_2 e_2+s_3 e_3$ is the corresponding two dimensional vector.  We define the distance of two points $s, s'\in T$, denoted by $d(s,s')$, the Euclidean distance of $C(s)$ and $C(s')$, which can be rewritten as
\begin{align}\label{dsspdefeq}
&d(s,s'):=\|(s_1-s_1') e_1+(s_2-s_2') e_2+(s_3-s_3') e_3\|\\
\nonumber&=\sqrt{\left<(s_1-s_1') e_1+(s_2-s_2') e_2+(s_3-s_3') e_3, (s_1-s_1') e_1+(s_2-s_2') e_2+(s_3-s_3') e_3\right>}\\
\nonumber&=\frac{1}{\sqrt{2}} \sqrt{(s_1-s_1')^2+(s_2-s_2')^2+(s_3-s_3')^2}.
\end{align}

The main triangle $T$ is divided into 4 \emph{subtriangles}, $T_i:=\{(s_1,s_2,s_3): s_i\in (1/2,1], 0\le s_1,s_2,s_3\le 1, s_1+s_2+s_3=1\}$ for $i=1,2,3$, and $T_4:=\{(s_1,s_2,s_3): 0\le s_1,s_2,s_3\le 1/2, s_1+s_2+s_3=1\}$. Then it is easy to see that the these are equilateral triangles, with centroid $C_i$ for $T_i$, for $1\le i\le 4$. We define the distance between a point $s\in T$ to the closes of the centres $C_1,\ldots, C_4$ as 
\begin{equation}\label{dCsdefeq}
d_C(s):=\min(d(s,C_1),d(s,C_2),d(s,C_3),d(s,C_4)).
\end{equation}
Then one can  see that for every $1\le i\le 4$, $s\in T_i$, we have $d_C(s)=d(s,C_i)$.
\begin{figure}
\centering
\includegraphics[height=5cm, bb=0 -1 498 456]{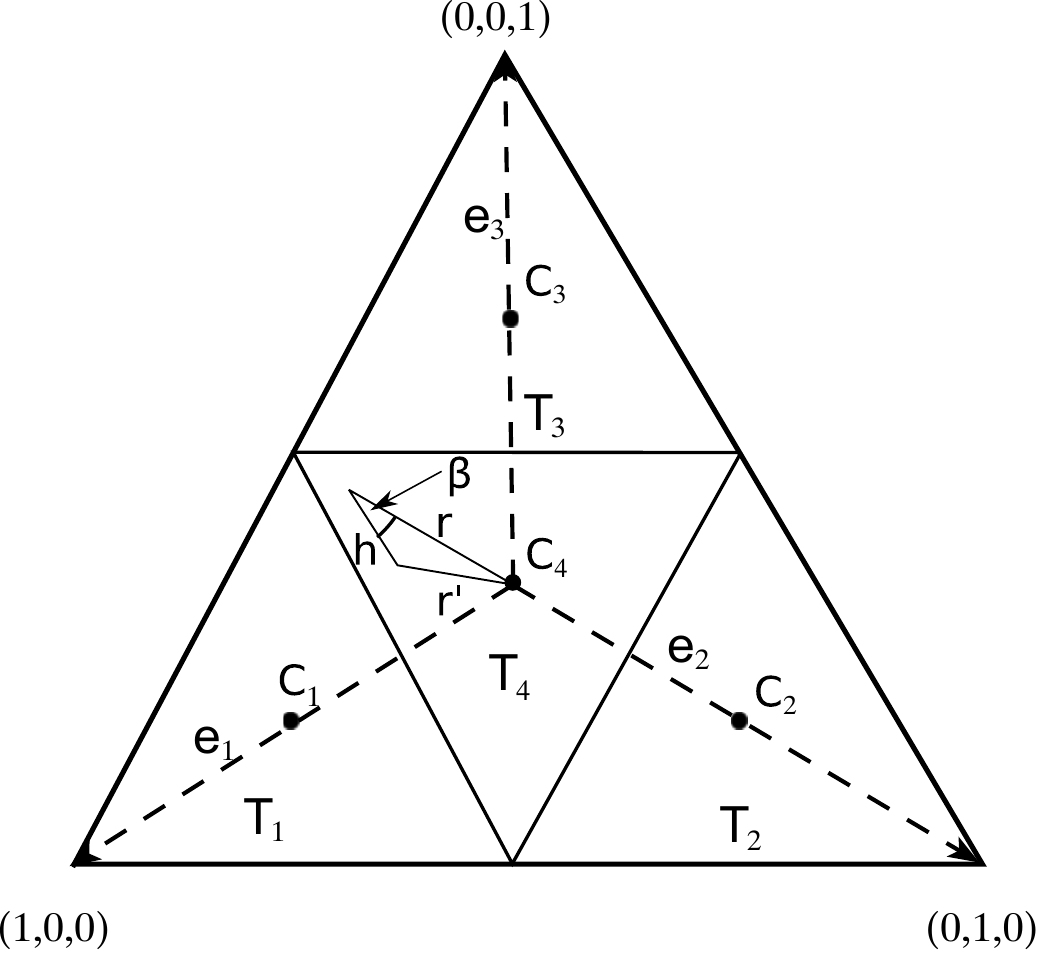}
\caption{Position of modes in barycentric coordinates}
\label{fig4modes}
\end{figure}

We define the modes in the following way. For $0\le j\le 10^7$, we set $m(j):=1$, and $F_{j}^{(1)}:=\Omega$. For $10^7\le j\le M$, we set $m(k):=4$, and 
\begin{align}\label{eqdefmodes1}
F_j^{(i)}&:=\left\{x\in \Omega: \frac{\sum_{k=1}^{j}\II[x_k=i]}{j}>\frac{1}{2}\right\}\text{ for }i =1, 2, 3,\text{ and}\\
\label{eqdefmodes2}F_j^{(4)}&:=\Omega\setminus\left(F_j^{(1)}\cup F_j^{(2)}\cup F_j^{(3)}\right).
\end{align}
Thus we compute the ratio of the spins of each colour among the first $k$ spins, look at which triangle this ratio vector falls into on Figure \ref{fig4modes}, and assign them to the corresponding mode. Now we proceed with the definition of the inner regions. For $j\le 10^7$, choose a single inner region as $\I_j^{(1)}:=\Omega$, and the border region is an empty set. For $10^7<j\le M$, and $1\le i\le 4$, we define the inner regions as 
\begin{equation}\label{innerregiondefeq}
\I_j^{(i)}:=\{\sigma\in \Omega: -\frac{\rho}{4}\le s_k(\sigma_{1:j})-C_{i,k}\le \frac{\rho}{2} \text{ for }k=1,2,3\} \quad\text { with }\quad \rho:=10^{-6},
\end{equation}
where $s_k(\sigma_{1:j}):=(\sum_{l=1}^{j} 1_{[\sigma_l=k]})/j$ is the ratio of spins of colour $k$ among the first $j$ spins. The points $(s_1(\sigma_{1:j}), s_2(\sigma_{1:j}),s_3(\sigma_{1:j}))$ for $\sigma\in \I_j^{(i)}$ fall in a small equilateral triangle centered at $C_i$ with sides parallel to the sides of $T$. The fact that $j>10^7$ ensures that the inner regions are non-empty. The border regions are defined as $\B_{j}^{(i)}:=F_j^{(i)}\setminus \I_j^{(i)}$ for $1\le i\le 4$.

\subsection{Key propositions and proof of Theorem \ref{thmvarPotts}}\label{seckeypropsandproof}
In this section, we state 5 key propositions bounding the various terms in Theorems \ref{thmwithmixing1} and \ref{thmwithmixing2}, and then prove  Theorem \ref{thmvarPotts} based on them. Our first proposition bounds the maximal density ratio constant $\Gamma_g$.
\begin{proposition}\label{keyprop1}
The maximal density constant $\Gamma_g:=\max_{0\le k\le n-1}\max_{x\in \Omega} \frac{\mmu_{k+1}(x)}{\mmu_{k}(x)}$ satisfies that $\Gamma_g\le \exp(2\tbeta_c)$.
\end{proposition}
\begin{proof}
Let us denote the number of spins of colours $1,2,3$ among $x_1,\ldots,x_k$ by $n_1(x_{1:k})$, $n_2(x_{1:k})$ and $n_3(x_{1:k})$, respectively. Then 
\begin{align*}
\frac{\mmu_{k+1}(x)}{\mmu_{k}(x)}=\frac{Z_{k}(\tbeta_c)}{Z_{k+1}(\tbeta_c)}
\cdot \exp\Bigg(\tbeta_c\bigg(&\frac{n_1^2(x_{1:k+1})+n_2^2(x_{1:k+1})+n_3^2(x_{1:k+1})}{k+1}\\
&-\frac{n_1^2(x_{1:k})+n_2^2(x_{1:k})+n_3^2(x_{1:k})}{k}\bigg)\Bigg).\end{align*}
Now it is easy to show that 
\[\left|\frac{n_1^2(x_{1:k+1})+n_2^2(x_{1:k+1})+n_3^2(x_{1:k+1})}{j+1}-\frac{n_1^2(x_{1:k})+n_2^2(x_{1:k})+n_3^2(x_{1:k})}{k}\right|\le 1,\]
thus 
\[\frac{Z_{k}(\tbeta_c)}{Z_{k+1}(\tbeta_c)}\exp(-\tbeta_c)\le \frac{\mmu_{k+1}(x)}{\mmu_{k}(x)}\le \frac{Z_{k}(\tbeta_c)}{Z_{k+1}(\tbeta_c)}\exp(\tbeta_c).\]
The fact that $\sum_{x\in \Omega}\mmu_{k}(x)=\sum_{x\in \Omega} \mmu_{k+1}(x)=1$ implies that $\frac{Z_{k}(\tbeta_c)}{Z_{k+1}(\tbeta_c)}\le \exp(\tbeta_c)$, and thus $\Gamma_g\le \exp(2\tbeta_c)$.
\end{proof}
Our second result bounds the growth-within-mode constants $B_{j,j+1}$ defined in \eqref{growthwithinmodedefeq}.
\begin{proposition}[Bounds on the growth-within-mode constants]\label{keyprop2}
\hspace{0mm}\\We have $B_{0,1}=B_{1,2}=1$, and for any $2\le j\le M-1$,
\[B_{j,j+1}\le 1+C\cdot \frac{\log(j)^{5}}{j^{3/2}}\]
for some absolute constant $C>0$.
\end{proposition}
The proof of this result, based on Taylor expansions, is quite technical, with no probabilistic ingredient, so it is included in Section \ref{secproofpropgrowthwithinmodePotts} of the Appendix.
%Figure \ref{figgrowthwithinthemode} illustrates the ratio of $(B_{j,j+1}-1)/(\frac{\log(j)^{5}}{j^{3/2}})$ as a function of $j$ from $j=50$ to $j=1000$ for $M=1001$. This ratio seems to be bounded by a finite constant, in accordance to the claim of the proposition.
%\begin{figure}
%\centering
%\includegraphics[width=14cm]{growthwithinmode2.pdf}
%\caption{Growth-within-mode constants: $(B_{j,j+1}-1)/(\frac{\log(j)^{5}}{j^{3/2}})$ as a function of $j$}
%\label{figgrowthwithinthemode}
%\end{figure}
The third proposition bounds the time needed to approach one of the centers. The proof is included in Section \ref{secproofprop2}.
\begin{proposition}[Time to get to a central region]\label{keyprop4}
Let $\{S(t)\}_{t\ge 0}$ be the magnetisation chain on $\Omega^S$. Let 
\[\tau:=\inf\left\{k\in \N: d_C(S(k))\le \frac{\rho}{8}\right\},\]
that is, the first time we get closer than $\frac{\rho}{8}$ to one of the centers. Then for any initial position $s\in \Omega^{S}$, any $r\in \R_+$,
\[\PP(\tau> r\cdot C M\log(M)|S(0)=s)\le \exp\left(-\lfloor r\rfloor\right),\]
where $C>0$ is an absolute constant.
\end{proposition}
The fourth proposition bounds the total variational distance from the local distribution in the modes. The proof is included in Section \ref{secproofprop3}.
\begin{proposition}\label{keyprop5}
Suppose that $k\ge 330M\log(M)$. Let $P$ denote the Markov kernel of the Glauber dynamics on $\mmu_M$. Then for $M$ larger than some absolute constant, for any $\sigma\in 
\I_M^{(i)}$ for some $1\le i\le 4$, we have
\[
\dtv(\P^k(\sigma, \cdot), \mmu_M^{(i)})\le \frac{C}{M^2}+3k^2 \exp(-C'\cdot M),
\]
where $C,C'>0$ are some absolute constants.
%and $\Cesc>0$ is the constant defined in Proposition \ref{keyprop3}.
\end{proposition}

Now we state the proof of our variance bound based on these key propositions.
\begin{proof}[Proof of Theorem \ref{thmvarPotts}]
By \eqref{eq.asymp.var.expansion}, the asymptotic variance satisfies that $V_M(f)=$\\ $\sum_{j=0}^{M} V_{j,M}(f)$, with the terms $V_{j,M}(f)$ can be bounding by Theorem \ref{thmwithmixing1} as
\begin{align}\nonumber
V_{j,M}(f)
&\leq
\norm{f}^2_\infty \Gamma_g \,
\prod_{i=j+1}^{M-1}
\curBK{ B_{i,i+1}  + \Gamma_g \, \vertiii{\KK_{i} - \widehat{\ppi}_i}_\infty }.\\
&\leq \label{eqVjnwithmixingPotts}
\norm{f}^2_\infty \Gamma_g \,
\prod_{i=0}^{M-1}B_{i,i+1}\cdot 
\prod_{i=j+1}^{M-1}
\curBK{1  + \Gamma_g \, \vertiii{\KK_{i} - \widehat{\ppi}_i}_\infty }.
\end{align}
By Proposition \ref{keyprop1}, we have $\Gamma_g\le \exp(2\tbeta_c)$. Proposition \ref{keyprop2} implies that 
$\prod_{i=0}^{M-1}B_{i,i+1}\le C_B$ for some absolute constant $C_B<\infty$. 

Let us choose $t_i=\lceil R i \log(i)^2\rceil$ for some constant $R$, and $\KK_i:=\P_i^{t_i}$, where $\P_i$ is the Markov kernel described in Section \ref{secinterpolationindependence} (combining a Glauber dynamics step in the first $i$ coordinates $(\sigma_1,\ldots, \sigma_i)$ with respect to $\mmu_{\tbeta_c,k}^{\mathrm{Potts}}$ and replacing the rest of the coordinates $(\sigma_{i+1},\ldots, \sigma_M)$ by independent copies).
Based on Theorem \ref{thmwithmixing2}, we have
\begin{equation} \label{Ptpitboundpitts}
\begin{aligned}
\vertiii{\KK_{i} - \widehat{\ppi}_i}_\infty=\vertiii{ \P_i^{t_i} - \widehat{\ppi}^{(t_i)} }_{\infty}
&\leq 
\max_{x \in \B_{i}} \; \PP \BK{ \SB(x,\lfloor t_i/2\rfloor) } \\
&+2 \, \max_{1\leq l \leq 4} \, \max_{\lceil t_i/2\rceil \leq r \leq t_i}\, \sup_{x \in \I^{(l)}} \; \dtv \BK{ \P_i^{r}(x,\cdot),\mmu^{(l)}},
\end{aligned}
\end{equation}
where $\B_{i}:=\cup_{l}\B_{i} ^{(l)}$ is the union of the border regions. By using Propositions \ref{keyprop4} and \ref{keyprop5} (applied by substituting $M=i$ and using the fact that the rest of the spins are independent), by choosing $R$ sufficiently large, for $i$ larger than some absolute constant, we have
\begin{align}
&\max_{x \in \B_{i}} \; \PP \BK{ \SB(x,\lfloor t_i/2\rfloor) } \le \frac{1}{i^2},\text{ and}\\
&2 \, \max_{1\leq l \leq 4} \, \max_{\lceil t_i/2\rceil \leq r \leq t_i}\, \sup_{x \in \I^{(l)}} \; \dtv \BK{ \P_i^{r}(x,\cdot),\mmu^{(l)}}\le \frac{C'}{i^2},
\end{align}
for some absolute constant $C'<\infty$. The result now follows by \eqref{eqVjnwithmixingPotts}.
\end{proof}

\subsection{Preliminary results}\label{secpreliminaryPotts}
In this section, we will first prove some concentration results for sequences of random variables satisfying certain drift conditions. After these, we show a coupling argument for showing how good mixing in the magnetisation chain can be used to show good mixing in in the original Glauber dynamics chain. The proofs are included in Section \ref{Secprelimlemmaproof} of the Appendix.
\begin{lemma}[A lower bound on the exit time]\label{lemmaanticoncentration}
Assume that $(X_l)_{l\ge 0}$ is a sequence of random variables adapted to some filtration $(\F_l)_{l\ge 0}$. Suppose that
\begin{enumerate}
\item[(1)] $|X_{l+1} - X_{l}|\le R$ almost surely for every $l\ge 0$ for some absolute constant $R$,
\item[(2)] $\E(X_{l+1}-X_{l}|\F_l)\le \delta$ for some $\delta>0$,
\item[(3)] $X_0=x_0$ for some fixed constant $x_0\in \R$, and
\item[(4)] $\Var(X_{l+1}-X_l|\F_l)\ge v$ for every $t\ge 0$ for some absolute constant $v>0$.
\end{enumerate}

Suppose that $z\ge 12\frac{R}{\sqrt{v}}$, then 
\[\PP\left[\min_{0\le l\le 4(z+2R/\sqrt{v})^2} (X_l-x_0)\le -z\sqrt{v}+4(z+2R/\sqrt{v})^2\delta\right]\ge \frac{1}{6}.\]
\end{lemma}
\begin{remark}
This lemma quantifies the fact that if the variance of the jumps is always at least an absolute constant greater than 0, and the drift $\delta$ towards to right is sufficiently small, then after $\OO(r)$ steps, we will move to the left $\OO(\sqrt{r})$ with reasonable probability. 
\end{remark}
\begin{lemma}[Moving in a region of negative drift]\label{lemmamoving}
Assume that $(X_l)_{l\ge 0}$ is a sequence of random variables adapted to some filtration $(\F_l)_{l\ge 0}$. Suppose that 
\begin{itemize}
\item[(1)] $|X_{l+1} - X_{l}|\le R$ almost surely for every $l\in \N$ for some absolute constant $R$,
\item[(2)] $\E(X_{l+1}-X_{l}|\F_l)\le -\delta$ for some $\delta>0$,
\item[(3)] $X_0=x_0$ for some fixed constant $x_0\in \R$.
\end{itemize}
Then the probabilities of moving backward, and forward, respectively, can be bounded as
\begin{align}
\label{eqmovingbackward}&\PP\left(X_{\lceil(1+c)T/\delta\rceil}-x_0\ge -T\right)\le \exp\left(-\frac{c T \delta}{4 R^2}\right)\text{ for any }c\ge 1, T> 0,\\
\label{eqmovingforward}&\PP\left(\max_{0\le k\le l}(X_k-x_0)\ge T\right)\le l\exp\left(-\frac{T\delta}{R^2}\right) \text{ for any }l\in \N, T> 0.
\end{align}
\end{lemma}
 
The following lemma shows that once two chains have met in magnetisation, they can be coupled together in $\OO(M\log(M))$ time with high probability.
\begin{lemma}[From coupling in magnetisation to coupling in spins]\label{lemmacinmagtocinspinsPotts}
Let $\sigma(0), \tsigma(0)\in \Omega$ such that $s(\sigma(0))=s(\tsigma(0))$. Let $(\sigma(t))_{t\ge 0}$ and $(\tsigma(t))_{t\ge 0}$ be two Glauber dynamics chains with temperature parameter $\tbeta_c$ started at $\sigma(0)$ and $\tsigma(0)$, respectively. 
Let $\tau:=\inf\{t\ge 0: \sigma(t)=\tsigma(t)\}$. Then there is a coupling $((\sigma(t))_{t\ge 0},(\tsigma(t))_{t\ge 0})$ such that for this coupling,
\[\PP(\tau>t)\le \frac{M}{2} \exp\left(-\frac{t}{9M}\right).\]
\end{lemma}

\subsection{Bounding the drift towards the centers}\label{secdrift}
The following lemma shows bounds the drift towards the centers at a given distance from the centers.
\begin{lemma}[Drift bound]\label{worstdriftlemma}
Let $\{S(k)\}_{k\ge 0}$ be the magnetisation chain on $\Omega^S$.
Then for any $s\in \Omega^{S}$ with $d_C(s)>1/M$, we have
\[\E( d_C(S(1))-d_C(S(0))|S(0)=s)\le -\frac{\varphi(d_C(s))}{M}+\frac{1}{M^2}\cdot \left(8+\frac{1}{2(d_C(s)-1/M)}\right),\]
with $\phi:\left[0,\sqrt{3}/6\right]\to \R$ defined as
\[\varphi(t):=
\left\{\begin{array}{l}
0.002\cdot \frac{t}{\sqrt{3}/24} \text{ for } 0\le t\le \sqrt{3}/24,\\
0.002\cdot \frac{\sqrt{3}/12-t}{\sqrt{3}/24} \text{ for } \sqrt{3}/24\le t\le \sqrt{3}/12\text{ and }\\
0.002\cdot \frac{t-\sqrt{3}/12}{\sqrt{3}/24} \text{ for } \sqrt{3}/12\le t \le \sqrt{3}/6.\\
\end{array}
\right.
\]
\end{lemma}
The proof of Lemma \ref{worstdriftlemma} is included in the Appendix.
Figure \ref{varphiplotfig} plots $\varphi(t)$ for $0\le t\le \sqrt{3}/6$. Note that $\varphi(t)=0$ for $t=0$ and $t=\sqrt{3}/12$.

\begin{figure}[h!]
  \centering
  \includegraphics[height=4cm]{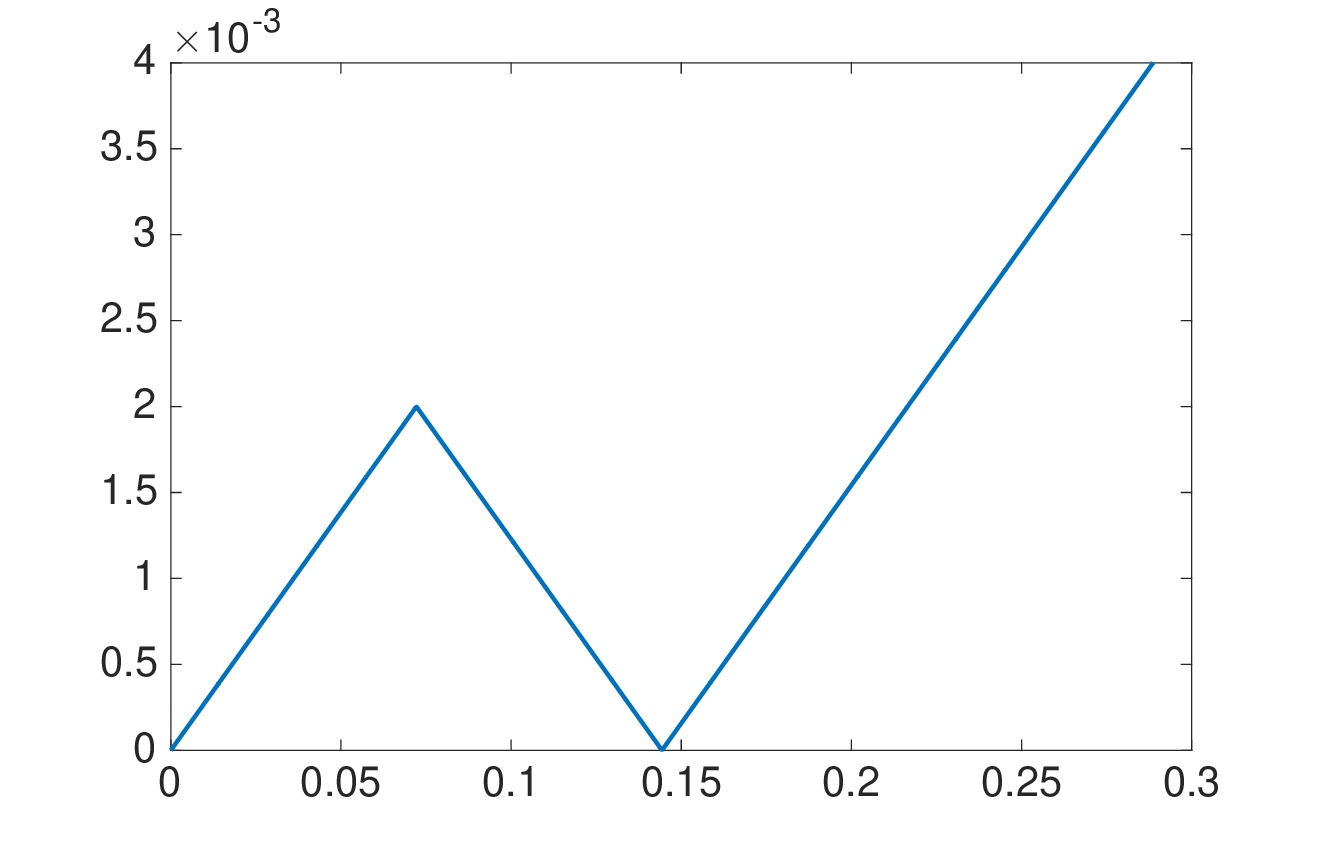}
    \caption{$\varphi(t)$ for $0\le t\le \sqrt{3}/6$}
    \label{varphiplotfig}
\end{figure}

\subsection{Bounds on escaping from the center}\label{secproofprop1}
In this section, we prove the following proposition bounding the probability of escaping from the region near one of the centers.
\begin{proposition}[Escaping from a central region]\label{keyprop3}
Let $\{S(t)\}_{t\ge 0}$ be the magnetisation chain on $\Omega^S$. Assuming that $S(0)=s$ with $d_C(s)\le \rho\cdot\frac{\sqrt{3}}{2}$, the probability of getting further away than $\rho\sqrt{3}$ in $l$ steps is bounded as
\begin{equation}\PP\left(\max_{0\le k\le l}d_C(S(k))>\rho\sqrt{3}|S(0)=s\right)\le l^2\exp(-\Cesc M),
\end{equation}
for $M$ larger than some absolute constant, where $\Cesc$ is an absolute constant that can be chosen as $\Cesc=0.005\rho^2$.
\end{proposition}
\begin{remark}
By the inner mode sets $\I^{(i)}_M$ and the sets $\tilde{\Lambda}^{(i)}:=\{\sigma\in \Omega: -\rho\le s_l(\sigma)-C_{i,l}\le 2\rho \text{ for }l=1,2,3\}$ satisfy that in the magnetisation space, their defining equations restrict them into equilateral triangles, with edge lengths $\frac{\sqrt{3}}{2}\cdot \rho$ for $\I^{(i)}_M$, and $2\sqrt{3}\cdot \rho$ for $\tilde{\Lambda}^{(i)}$. Therefore, in particular, this proposition implies that if we start from $\sigma(0)=\sigma\in \I^{(i)}_M$, then the probability of the Glauber dynamics chain $\{\sigma(k)\}_{k\ge 0}$ exiting from $\tilde{\Lambda}^{(i)}$ in the first $l$ steps can be bounded as
\begin{equation}\label{exitfromLambdaieq}\PP\left(\sigma(k)\notin  \tilde{\Lambda}^{(i)} \text{ for some }1\le k\le l |\sigma(0)=\sigma\right)\le l^2\exp(-\Cesc M),
\end{equation}
for $M$ larger than some absolute constant.
\end{remark}
\begin{proof}[Proof of Proposition \ref{keyprop3}]
Using Lemma \ref{worstdriftlemma}, we can see that for $s\in \Omega^S$ satisfying that $d_C(s)\in \left[\rho\frac{\sqrt{3}}{2},\rho\sqrt{3}\right]$, for $M$ larger than some absolute constant,
\[\E( d_C(S(1))-d_C(S(0))|S(0)=s)\le-\frac{\phi\left(\frac{\rho \sqrt{3}}{2}\right)}{2M}\le \frac{-0.012\rho}{M}.\]
In order to get further away than $\rho$ from one of the centers, we need to spend a period of time when the distance is between $\rho\frac{\sqrt{3}}{2}$ and $\rho\sqrt{3}$, and then exceed $\rho\sqrt{3}$. 

Notice that we cannot apply the martingale-type inequalities of Lemma \ref{lemmamoving} directly to $\{d_C(S(k))\}_{k\ge 0}$ because the drift does not hold uniformly in every $k$. Instead of  direct application, we use a coupling argument. For every $0\le k\le l-1$, we define a sequence of random variables $D^{(k)}_0, D^{(k)}_{1}, \ldots, $ as follows. First, we set $D^{(k)}_0:=d_C(S(k))$. After this, we define the rest of the sequence such that it satisfies that for every $j\in \N$,
\begin{align*}D^{(k)}_{j+1}-D^{(k)}_{j}&:=1_{\left[d_C(S(k+j))\in \left[\rho\frac{\sqrt{3}}{2},\rho\sqrt{3}\right] \left[\rho\frac{\sqrt{3}}{2},\rho\sqrt{3}\right] \right]}\cdot [d_C(S(k+j+1))-d_C(S(k+j))]\\
&\quad - 1_{\left[d_C(S(k+j))\notin \left[\rho\frac{\sqrt{3}}{2},\rho\sqrt{3}\right]\right]}\cdot \frac{0.012\rho}{M}.
\end{align*}
From the above definitions it follows that for $s\in \Omega^S$ satisfying that $d_C(s)\in \left[\rho\frac{\sqrt{3}}{2},\rho\sqrt{3}\right]$,
\[\PP\left(\left.\max_{0\le k\le l}d_C(S(k))>\sqrt{3}\rho\right|S(0)=s\right) \le \sum_{k=0}^{l-1}\PP\left(\left. \max_{0\le i\le l-k}(D^{(k)}_{i})\ge \sqrt{3}\rho \right| S(0)=s\right).\]
Then from their definition, we can see that the random variables $(D^{(k)}_{i})_{0\le i\le {l-k}}$ satisfy the conditions of Lemma \ref{lemmamoving} with $R=\frac{1}{M}$ and $\delta:=\frac{0.012\rho}{M}$, and the claim of the proposition follows by applying \eqref{eqmovingforward} with $T=\frac{\sqrt{3}}{2}\rho-\frac{1}{M}$ and $\delta=\frac{0.012\rho}{M}$ on $\max_{0\le i\le l-k}(D^{(k)}_{i})$, and then summing up.
\end{proof}

\subsection{Getting to one of the centers}\label{secproofprop2}
In this section, we prove Proposition \ref{keyprop4} based on drift arguments and concentration inequalities. The following lemma will be used for the proof.
\begin{lemma}[Minimum variance of jumps]\label{lemmaminvar}
Let $\{S(t)\}_{t\ge 0}$ be the magnetisation chain on $\Omega^S$. Then for $M$ larger than some absolute constant, for any starting point $s\in \Omega^S$, 
\[\Var(d_C(S(1))|S(0)=s)\ge \frac{v_{\min}}{M^2}, \quad \text{ with }\quad v_{\min}:=0.001.\]
\end{lemma}
\begin{proof}
Using the notations of the proof of Lemma \ref{worstdriftlemma}, it is straightforward to show that
\begin{align*}P_{i\to j}(s_1,s_2, s_3)\ge s_i\cdot \frac{1}{2+\exp(\tbeta_c)}=\frac{s_i}{18}.
\end{align*}
Moreover, it is also easy to check that
\begin{align*}
P_{\circlearrowleft}(s_1,s_2, s_3)&\ge \sum_{1\le i\le 3}\frac{s_i}{1+\exp\left[\frac{2}{M}+2\tbeta_c(s_j-s_i)\right]+\exp\left[\frac{2}{M}+2\tbeta_c(s_k-s_i)\right]}\ge \frac{1}{54}.
\end{align*}
The proof is based on the fact that for an equilateral triangle of unit edge length, and a point on the plane outside the triangle, the difference between the distances of the point and the closest and furthest away corners of the triangle is at least $\frac{\sqrt{3}}{2}-\frac{1}{2}$.

Assume first that $(s_1,s_2,s_3)$ more than $\frac{1}{M}$ away from the edges of the central triangle $T_4$ (in $d$ distance). Then without loss of generality, assume that $s_1\ge \frac{1}{3}$. Then 
$P_{1\to 2}(s_1,s_2, s_3)\ge \frac{1}{54}$ and $P_{1\to 3}(s_1,s_2, s_3)\ge \frac{1}{54}$. Since the three positions $s$, $s^{1\to 2}$ and $s^{1\to 3}$ form an equilateral triangle of side length $\frac{1}{M}$, unless a central point is included in this triangle, there is at least  $\frac{1}{M}\left(\frac{\sqrt{3}}{2}-\frac{1}{2}\right)$ difference between two of $d_C(s)$, $d_C(s^{1\to 2})$ and $d_C(s^{1\to 3})$. Therefore, the variance is lower bounded as
\[\Var(d_C(S(1))|S(0)=s)\ge \frac{2}{54}\left(\frac{1}{2}\cdot \left(\frac{\sqrt{3}}{2}-\frac{1}{2}\right)\cdot \frac{1}{M}\right)^2> \frac{0.001}{M^2}.\]
If the triangle formed by $s$, $s^{1\to 2}$ and $s^{1\to 3}$ contains a central point, then one can check that the same bound still holds for $M$ greater than some absolute constant by suitably choosing the direction $i\to j$.

Finally, if $(s_1,s_2,s_3)$ is no more than $\frac{1}{M}$ away from the edges of the central triangle $T_4$, then one can still choose $i\to j$ in a way that we do not exit from the triangle that we are in ($T_1$, $T_2$, $T_3$ or $T_4$), for $M$ larger than some absolute constant, $|d_C(s^{i\to j})-d_C(s)|>\frac{0.4}{M}$, and $s_i\ge \frac{1}{4}$. For this choice, we have $P_{i\to j}(s_1,s_2, s_3)\ge \frac{1}{72}$, and thus
\[\Var(d_C(S(1))|S(0)=s)\ge \frac{2}{72}\left( \frac{1}{2}\cdot \frac{0.4}{M}\right)^2> \frac{0.001}{M^2}.\qedhere\]
\end{proof}

Now we are ready to prove the main result of this section.
\begin{proof}[Proof of Proposition \ref{keyprop4}]
Let $c_1:=1000$ and $c_2:=1$. Let
\begin{align*}r&:= \left\lceil \left(\frac{\sqrt{3}}{12}-c_1 \sqrt{\frac{\log(M)}{M}}\right)/\left(\frac{c_2}{\sqrt{M}}\right)\right\rceil+1, \text{ and }\\
m&:= r+1+\left\lceil \left(\frac{\sqrt{3}}{12}-c_1 \sqrt{\frac{\log(M)}{M}}-\frac{\rho}{8}\right)/\left(\frac{c_2}{\sqrt{M}}\right)\right\rceil.
\end{align*}
Define a sequence of distances $d_0>d_1> \ldots> d_{m}$ as follows,
\begin{align*}
&d_0:=\frac{\sqrt{3}}{6}, \quad d_{r-1}:=\frac{\sqrt{3}}{12}+c_1\sqrt{\frac{\log(M)}{M}},  \quad
d_{r}:=\frac{\sqrt{3}}{12},  \quad d_{r+1}:=\frac{\sqrt{3}}{12}-c_1\sqrt{\frac{\log(M)}{M}}, \\ \\
&d_{m}:=\frac{\rho}{8}, \quad d_k-d_{k+1}:=\frac{c_2}{\sqrt{M}} \text{ for }k\in \{1,\ldots, m-2\}\setminus \{r-1,r\}.
\end{align*}
For $1\le j\le m$, we define the arrival times 
$\tau_j:=\inf\{k\ge 0: d_C(S(k))\le d_j\}$. The proof will consist of subsequently estimating the differences $\tau_{j+1}-\tau_j$ for $0\le j\le m-1$.
Figure \ref{figddistances} illustrates the position of the distances $(d_j)_{0\le j\le m}$, and the function $\varphi(t)$.
\begin{figure}[h!]
  \centering
  \includegraphics[height=4.3cm]{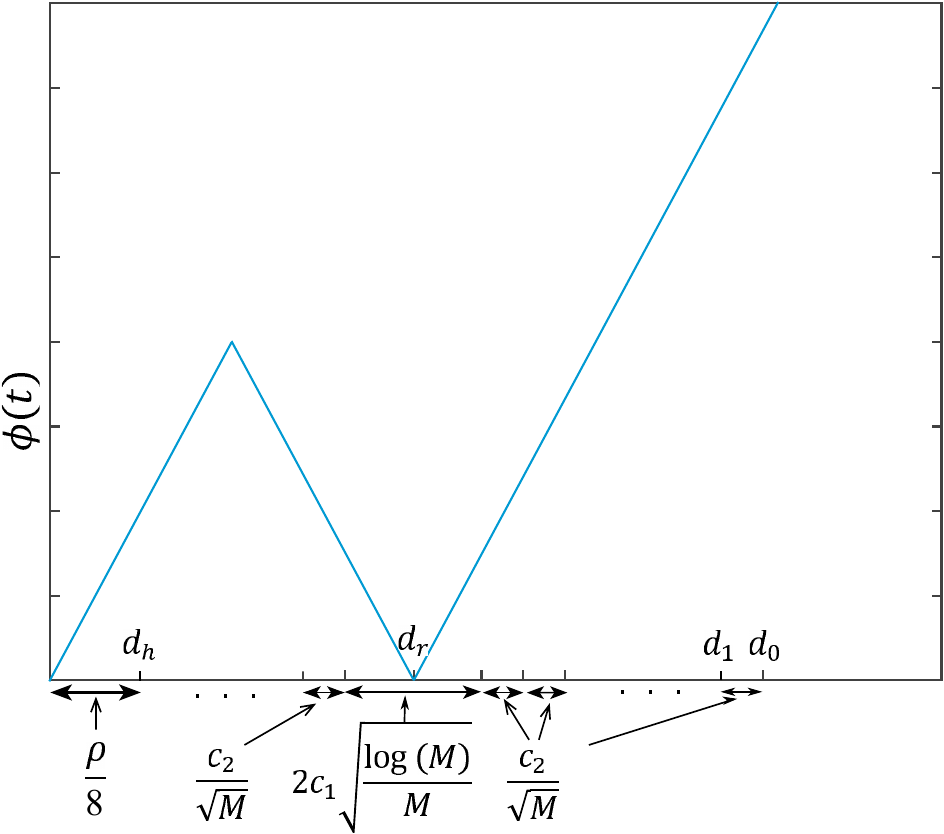}
    \caption{The distances $d_0, d_1, \ldots, d_m$.}
    \label{figddistances}
\end{figure}

Based on Lemma \ref{worstdriftlemma}, we obtain the following simplified drift bounds. For $M$ larger than some absolute constant, for every $s\in \Omega^S$, we have
\begin{align}
\E(d_C(S(1))-d_C(S(0))|S(0)=s)&\le \frac{20}{M^2}.\label{eqdriftb1}
\end{align}
Based on the definition of the distances $d_j$ and the function $\phi(t)$, it follows that for $M$ larger than some absolute constant, for every $j\in \{0,\ldots, m-1\}\setminus \{r, r-1\}$, for every $s\in \Omega^S$ such that 
$d_C(s)\in \left[d_{j+1}-\frac{1}{3}c_1\sqrt{\frac{\log(M)}{M}},d_{j}+\frac{1}{3}c_1\sqrt{\frac{\log(M)}{M}}\right]$, we have
\begin{equation}\label{eqdriftb2}
\E\left(\left. d_C(S(1))-d_C(S(0))\right|S(0)=s\right)\le -\frac{\phi(d_{j})}{2M}.
\end{equation}

First, we are going to estimate $\tau_{r+1}-\tau_{r-1}$. By applying Lemma \ref{lemmaanticoncentration} to $\{d_C(S(\tau_{r-1}+k))\}_{k\ge 0}$ with $\delta= \frac{20}{M^2}$, $R=\frac{1}{M}$, $v=\frac{v_{\min}}{M^2}$ (based on Lemma \ref{lemmaminvar}) and $z=4c_1\sqrt{\frac{\log(M)}{M}}/\sqrt{v}=\frac{4c_1 }{\sqrt{v_{\min}}}\sqrt{M\log(M)}$, it follows that for $M$ larger than some absolute constant, and any $s\in \Omega^{S}$,
\begin{equation}\label{taueq1}
\PP(\tau_{r+1}-\tau_{r-1}\le R_{r-1}+R_r|S(0)=s)\ge c_4,
\end{equation}
where $R_{r-1}=R_{r}:=c_3 M \log(M)$,  $c_3:= \frac{70 c_1^2 }{v_{\min}^2}$ and $c_4:=\frac{1}{6}$.

For $j\in \{0,\ldots, m-1\}\setminus \{r-1,r\}$, let
\[R_j:=\left\lceil \left(2+\frac{32\log(M)}{c_2 M^{1/2} \varphi(d_{j})}\right)\cdot \frac{2c_2 \sqrt{M}}{\varphi(d_{j})}\right\rceil.\]
We are going to show that the probability of $\tau_{j+1}-\tau_{j}$ being greater than $R_j$ is small. 
Let us define the interval $I_j$ as
\[I_j:=\left[d_{j+1}-\frac{1}{3}c_1\sqrt{\frac{\log(M)}{M}},d_{j}+\frac{1}{3}c_1\sqrt{\frac{\log(M)}{M}}\right].\]
Notice that the drift bound \eqref{eqdriftb2} only holds for positions $s$ for which $d_C(s)$ is within interval $I_j$. Since we can possibly exit from this interval within $R_j$ steps from time $\tau_j$, we cannot apply the martingale inequalities directly to $\{d_C(S(\tau_{j}+k))\}_{k\ge 0}$ as previously.
This difficulty can be resolved by a coupling argument similar to the one in the proof of Proposition \ref{keyprop3}.

For every $j\in \{0,\ldots, m-1\}\setminus \{r-1,r\}$, we define a sequence of random variables $D^{(j)}_0, D^{(j)}_{1}, \ldots, $ as follows. First, we set $D^{(j)}_0:=d_C(S(\tau_{j}))$. After this,
we define the rest of the sequence based on the condition that for every $l\in \N$,
\[D^{(j)}_{l+1}-D^{(j)}_{l}:=1_{[d_C(S(\tau_{j}+l))\in I_j ]}\cdot [d_C(S(\tau_{j}+l+1))-d_C(S(\tau_{j}+l))]- 1_{[d_C(S(\tau_{j}+l))\notin I_j ]}\cdot \frac{\phi(d_{j})}{2M}.\]
Due to the definition of this sequence, for every $s\in \Omega^S$, we have 
\begin{align}
&\nonumber\PP(\tau_{j}-\tau_{j-1}>R_j|S(0)=s)\\
&\le \PP\left(\left.\max_{0\le i\le R_j}D^{(j)}_0 >d_j+ \frac{1}{3}c_1\sqrt{\frac{\log(M)}{M}}\right|S(0)=s\right)+ \PP\left(\left.D^{(j)}_{R_j}> d_{j+1}\right|S(0)=s\right).\label{Djcouplingineq}
\end{align}
Based on the definition, it follows that $\left\{D^{(j)}_i\right\}_{i\ge 0}$ satisfies the conditions of Lemma \ref{lemmamoving} with filtration $(\F_l)_{l\ge 0}:=(\sigma\left(S(\tau_{j}),\ldots, S(\tau_{j}+l)\right) )_{l\ge 0}$, $R=\frac{1}{M}$, and $\delta=\frac{\phi(d_{j})}{2M}$.
Therefore by \eqref{eqmovingbackward} and \eqref{eqmovingforward}, for $M$ larger than some absolute constant, for every $j\in \{0,\ldots, m-1\}\setminus \{r-1,r\}$, for every $s\in \Omega^S$, we have $\PP\left(\left.D^{(j)}_{R_j}> d_{j+1}\right|S(0)=s\right)\le \frac{1}{2M^2}$, and
\begin{align*}
%&\PP\left(\left.D^{(j)}_{R_j}> d_{j+1}\right|S(0)=s\right)\le \frac{1}{2M^2},\text{ and }\\
&\PP\left(\left.\max_{0\le i\le R_j}D^{(j)}_0 >d_j+ \frac{1}{3}c_1\sqrt{\frac{\log(M)}{M}}\right|S(0)=s\right)\le \frac{1}{2M^2}, 
\end{align*}
therefore
\begin{equation}\label{taueq2}
\PP(\tau_{j}-\tau_{j-1}>R_j|S(0)=s)\le \frac{1}{M^2}.
\end{equation}
Now it is easy to show that for $M$ larger than some absolute constant, 
\begin{equation}\label{Rsumeq}
\sum_{j=0}^{m}R_j\le c_5 M\log(M),
\end{equation}
for some absolute constant $c_5$. Moreover, based on \eqref{taueq1} and \eqref{taueq2}, by the union bound, it follows that for $M$ larger than some absolute constant, for every $s\in \Omega^S$,
\begin{equation}
\PP(\tau_{m}<c_5 M\log(M)|S(0)=s)>\frac{c_4}{2}. 
\end{equation}
Since this result holds for any $s\in \Omega^S$, therefore by repeated application, we obtain that for any $k\in \Z_+$,
\begin{equation}
\PP(\tau_{m}>k(c_5+1)M\log(M)|S(0)=s)\le \left(1-\frac{c_4}{2}\right)^k, 
\end{equation}
which implies the claim of the theorem.
\end{proof}

\subsection{Fast mixing at the center via curvature}\label{secproofprop3}
In this section, we are going to show that once we get near the center of one of the modes, we will quickly approach the stationary distribution restricted to that mode, proving Proposition \ref{keyprop5}. For this, we will use a curvature (i.e. path-coupling) argument.

Let $(\Lambda,d_{\Lambda})$ be a Polish metric space, and $P(x,y)$ a Markov kernel on $\Lambda$. 
For two probability measures $\mmu, \meta$ on $\Lambda$, we define $\Pi(\mmu,\meta)$ as the set of measures $\nu$ on $\Lambda\times \Lambda$ with first and second marginals $\mmu$ and $\meta$ (that is, $\int_{y\in \Lambda} \nu(dx, dy)=\mmu(dx)$ and $\int_{x\in \Lambda} \nu(dx, dy)=\meta(dy)$), and
we define the Wasserstein distance of $\mmu$ and $\meta$, denoted by $W_1(\mmu,\meta)$, as
\[W_1(\mmu,\meta):=\sup_{\nu \in \Pi(\mmu,\meta)} \int_{x,y\in \Lambda} d(x,y) \nu(dx, dy).\]

Then for $x, y\in \Lambda$, $x\ne y$, \cite{Ollivier2} defines 
\[\kappa(x,y):=\frac{d_{\Lambda}(x,y)-W_1(\P(x,\cdot), \P(y,\cdot))}{d_{\Lambda}(x,y)},\]
and calls $\kappa:=\inf_{x,y\in \Lambda, x\ne y}\kappa(x,y)$ the \emph{coarse Ricci curvature}. This quantity can be used to estimate the speed of convergence to the stationary distribution of 
the Markov chain.

\cite{Ollivier2} says that $(\Lambda, d_{\Lambda})$ satisfies the $\epsilon$-geodesic property if for any two points $x,y\in \Lambda$, there exists a set of points $x_0:=x, x_1,\ldots, x_l:=y$ such that
$d_{\Lambda}(x,y)=d_{\Lambda}(x_0,x_1)+\ldots+d_{\Lambda}(x_{l-1},x_l)$ and $d_{\Lambda}(x_i,x_{i+1})\le \epsilon$ for every $0\le i<l$.
Proposition 19 of \cite{Ollivier2} shows that for such spaces, $\kappa=\inf_{x,y\in \Lambda, x\ne y, d_{\Lambda}(x,y)\le \epsilon}\kappa(x,y)$,
that is, it suffices to estimate $\kappa(x,y)$ for pairs of points whose distance is at most $\epsilon$. We will choose $\Lambda$ as a subset of $\Omega^{S}$ (the state space of the colour ratio vector $S$), and define the distance 
\[d_{\Lambda}(x,y):=\frac{M}{2}\left(\left|x_1-y_1\right|+\left|x_2-y_2\right|+\left|x_3-y_3\right|\right),\]
corresponding to the amount of edges one needs to traverse on the hexagonal lattice to get from $x$ to $y$. This distance is $1$-geodesic, thus it will suffice to bound $\kappa(x,y)$ for neighbouring points $x, y$. It turns out that if we choose $\bm{P}$ as the Markov kernel of the magnetisation chain, $\kappa(x,y)$ will not be positive through the whole state space $\Omega^{S}$. This negative curvature still persists even for the local restrictions of the Markov kernel to the 4 modes (the 4 triangles on Figure \ref{fig4modes}). This is caused by the non-convexity of the distribution at the regions separating the modes. For $1\le i\le 4$, we define 4 regions near the center of the modes as
\begin{equation}\label{enlargedinnerregiondefeq}
\Lambda^{(i)}:=\{s\in \Omega^S: -\rho\le s_l-C_{i,l}\le 2\rho \text{ for }l=1,2,3\}.
\end{equation}
The following proposition shows that the curvature is positive in these regions. The proof of this result is included in Section \ref{secproofcurvaturebound} of the Appendix.

\begin{proposition}[Curvature bound]\label{propcurvature}
Consider the state space $(\Lambda^{(m)}, d_{\Lambda})$ as above for some $1\le m\le 4$, and let $\P^{\mathrm{mag}}_{(m)}$ be the restriction of the kernel of  the magnetisation chain $\P^{\mathrm{mag}}$ to $\Lambda^{(i)}$, i.e. for $x, y\in \Lambda^{(m)}$, we let
\[\P^{\mathrm{mag}}_{(m)}(x,dy):=\P^{\mathrm{mag}}\left(x,(\Lambda^{(m)})^c\right)\cdot \delta_x(dy)+1_{[y\in \Lambda^{(m)}]}\cdot \P(x,dy),\]
where $\delta_x(dy)$ corresponds to the Dirac-$\delta$ distribution at $x$. Let $\kappa$ denote the coarse Ricci curvature of the kernel $\P^{\mathrm{mag}}_{(m)}$ on the metric space $(\Lambda^{(m)}, d_{\Lambda})$. Then for $M$ larger than some absolute constant, we have $\kappa\ge \frac{0.01}{M}$.
\end{proposition}

Let $\tilde{\Lambda}^{(i)}:=\{\sigma\in \Omega: s(\sigma)\in \Lambda^{(i)}\}$. Let $\mmu_{M|\tilde{\Lambda}^{(i)}}$ denote the restriction of $\mmu_{M}$ to $\tilde{\Lambda}^{(i)}$ (i.e. $\mmu_{M|\tilde{\Lambda}^{(i)}}(f)=\frac{\mmu_{M}\left(f\cdot 1_{\tilde{\Lambda}^{(i)}}\right)}{\mmu_M(\tilde{\Lambda}^{(i)})}$, and $\mmu_{M}^{(i)}$ denote the restriction of $\mmu_M$ to $F_M^{(i)}$ (i.e. $\mmu_{M}^{(i)}(f)=\frac{\mmu_{M}\left(f\cdot 1_{F_M^{(i)}}\right)}{\mmu_M(F_M^{(i)})}$). The following lemma bounds the total variational distance of these two distributions. It will be used in the proof of Proposition \ref{keyprop5}.
\begin{lemma}\label{dtvLambdailemma}
For $M$ larger than some absolute constant, for every $1\le i\le 4$, we have
\begin{align}\label{propLambdaieq1}
\dtv\left(\mmu_{M|\tilde{\Lambda}^{(i)}}, \mmu_{M}^{(i)}\right)\le \frac{C_{\Lambda}}{M^2},\quad\text{ and }\quad\mmu_{M|\tilde{\Lambda}^{(i)}}\left(\tilde{\Lambda}^{(i)}\setminus \I_M^{(i)}\right)\le \frac{C_{\Lambda}}{M^2},
%\label{propLambdaieq2}
\end{align}
for some absolute constant $C_{\Lambda}<\infty$ (the inner regions $\I_M^{(i)}$ are defined as in \eqref{innerregiondefeq}).
\end{lemma}
This lemma essentially states that most of the mass of the distribution $\mmu_{M}$ is contained near the centers. The proof of is included in Section \ref{secproofdtvLambdailemma} of the Appendix. Now we are ready to prove Proposition \ref{keyprop5}.

\begin{proof}[Proof of Proposition \ref{keyprop5}]
Let $\P_{\Lambda^{(i)}}$ be a Markov kernel that is the restriction of $\P$ to $\Lambda^{(i)}$, i.e., for every $x,y\in \Lambda^{(i)}$,
\[\P_{\tilde{\Lambda}^{(i)}}(x,dy):=\P\left(x,(\Lambda^{(i)})^c\right)\cdot \delta_x(dy)+1_{[y\in \tilde{\Lambda}^{(i)}]}\cdot \P(x,dy).\]
Let $\sigma(0)$ be a fixed element of $\tilde{\Lambda}^{(i)}$, $\sigma'(0)\sim \mmu_{M|\tilde{\Lambda}^{(i)}}$, and define copies of them as $\Sigma(0):=\sigma(0)$ and $\Sigma'(0)=\sigma'(0)$. 
Let $\{\sigma(i)\}_{i\ge 0}$ and $\{\sigma'(i)\}_{i\ge 0}$ be two Markov chains evolving according to the kernel $P$, and let $\{\Sigma(i)\}_{i\ge 0}$ and $\{\Sigma'(i)\}_{i\ge 0}$ be two Markov chains evolving according to the kernel $P_{\tilde{\Lambda}^{(i)}}$. We are going to  obtain the total variational distance bound by creating a coupling $\{\sigma(i),\sigma'(i),\Sigma(i),\Sigma'(i)\}_{0\le i\le k}$ of these four chains. Let $\mmu^{\mathrm{mag}}_{(i)}$ denote the restriction of the magnetisation distribution $\mmu^{\mathrm{mag}}_{\tbetac,M}$ to $\Lambda^{(i)}$.
First, we note that for $k_1:=\lfloor 300M\log(M)\rfloor$, based on Proposition \ref{propcurvature}, and Corollary 21 of \cite{Ollivier2}, we have
\begin{align}&\label{dtvineq1}\dtv\left((\P^{\mathrm{mag}}_{(i)})^{k_1}(s(\sigma(0)),\cdot), \mmu^{\mathrm{mag}}_{(i)}\right)\le 
W_1((\P^{\mathrm{mag}}_{(i)})^{k_1}(s(\sigma(0)),\cdot), \mmu^{\mathrm{mag}}_{(i)})\\
&\le 
\nonumber 3 \rho M\cdot \left(1-\frac{0.01}{M}\right)^{k_1}\le \frac{1}{2M^2},
\end{align}
for $M$ larger than some absolute constant. In the first step we have used fact that the minimum distance between disjoint two points in our metric $d_{\Lambda}$ is $1$. 

Let $\bm{\nu}$ and $\meta$ be two probability measures on a finite space $W$. Proposition 4.7 of \cite{peresbook} shows the existence of an optimal coupling, i.e. a coupling $(X,Y)$ of two random variables $X\sim \bm{\nu}$ and $Y\sim \meta$ such that $\PP(X\ne Y)=\dtv(\bm{\nu},\meta)$.

We choose the coupling $(s(\Sigma(k_1)),s(\Sigma'(k_1)))$ as an optimal coupling. By \eqref{dtvineq1} means that they 
satisfy that
\begin{equation}\label{cplineq1}\PP(s(\Sigma(k_1))\ne s(\Sigma'(k_1)))\le \frac{1}{M^2}
\end{equation}
for $M$ larger than some absolute constant. Given the joint distribution of $(s(\Sigma(k_1)),s(\Sigma'(k_1)) )$, we choose the joint distribution $(\Sigma(k_1),\Sigma'(k_1))$ arbitrarily among the possibilities. After this, we define $\{\Sigma(i)\}_{1\le i\le k_1-1}$ and $\{\Sigma'(i)\}_{1\le i\le k_1-1}$ based on their conditional distribution given $\Sigma(0), \Sigma(k_1)$, and $\Sigma'(0), \Sigma'(k_1)$, respectively (their joint distribution can be chosen arbitrarily among the possibilities).

Now that $\{\Sigma(i),\Sigma'(i)\}_{0\le i\le k_1}$ is defined, we define $\{\sigma(i)\}_{0\le i\le k_1}$ and $\sigma'(i)\}_{0\le i\le k_1}$ recursively, based on the optimal coupling with $\{\Sigma(i)\}_{0\le i\le k_1}$ and $\{\Sigma'(i)\}_{0\le i\le k_1}$, respectively. That is, if we have already defined $\{\sigma(j)\}_{0\le j\le i}$ for some $0\le i\le k_1-1$, then we define $\sigma(i+1)$ such that $\sigma(i+1)$ and $\Sigma(i+1)$ are optimally coupled, and similarly for $\sigma'(i+1)$ and $\Sigma'(i+1)$. 
Due to the definition of the Markov kernels $\P_{\tilde{\Lambda}^{(i)}}$ and $\P$, we have 
\[\PP(\sigma(k_1)\ne \Sigma(k_1))\le \PP(\sigma(i)\notin \Lambda^{(i)} \text{ for some }1\le i\le k_1),\]
and since $\sigma(0)\in \I_M^{(i)}$, by Proposition \ref{keyprop3}, we have
\begin{equation}\label{cplineq2}\PP(\sigma(k_1)\ne \Sigma(k_1))\le k_1^2 \exp(-\Cesc M).\end{equation}
Based on Lemma \ref{dtvLambdailemma}, we have $\PP(\sigma'(0)\notin \I_M^{(i)})\le \frac{C_{\Lambda}}{M^2}$, and therefore by the same argument, we have
\begin{equation}\label{cplineq3}\PP(\sigma'(k_1)\ne \Sigma'(k_1))\le k_1^2 \exp(-\Cesc M)+\frac{C_{\Lambda}}{M^2}.\end{equation}
At this point, by combining \eqref{cplineq1}, \eqref{cplineq2} and \eqref{cplineq3}, we can see that
\begin{equation}\label{cplineq4}
\PP(s(\sigma_{k_1})\ne s(\sigma'_{k_1}))\le 2k_1^2 \exp(-\Cesc M) +\frac{C_{\Lambda}+1}{M^2}.
\end{equation}
From this point onwards, whenever $s(\sigma_{k_1})= s(\sigma'_{k_1})$, we define the joint distribution $\{\sigma_{i},\sigma'_{i}\}_{k_1\le i\le k}$ conditioned on $(\sigma_{k_1}, \sigma_{k_1}')$  as the coupling given by Lemma \ref{lemmacinmagtocinspinsPotts}. When $s(\sigma_{k_1})\ne s(\sigma'_{k_1})$, the joint distribution $\{\sigma_{i},\sigma'_{i}\}_{k_1\le i\le k}$ is chosen arbitrarily. 
Then based on Lemma \ref{lemmacinmagtocinspinsPotts}, for $M$ larger than some absolute constant, we have
\begin{align}\label{cplineq5}
\PP(\sigma_{k}\ne \sigma'_{k})&\le 2k_1^2 \exp(-\Cesc M) +\frac{C_{\Lambda}+1}{M^2}+\frac{M}{2}\exp\left(-\frac{(k-k_1)}{9M}\right)\\
\nonumber&\le 2k_1^2 \exp(-\Cesc M) +\frac{C_{\Lambda}+2}{M^2}.
\end{align}
Finally, we define the joint distribution of $\{\Sigma'(i), \sigma'(i)\}_{k_1\le i\le k}$ as the optimal coupling in each step as previously. With the same argument as in \eqref{cplineq3}, we have
\begin{equation}\label{cplineq6}\PP(\sigma'(k)\ne \Sigma'(k))\le k^2 \exp(-\Cesc M)+\frac{C_{\Lambda}}{M^2},\end{equation}
for $M$ larger than some absolute constant. By combining \eqref{cplineq5} and \eqref{cplineq6}, we obtain that for $M$ larger than some absolute constant,
\begin{equation}\label{cplineq7}
\PP(\sigma_{k}\ne \Sigma'_{k})\le 3k^2 \exp(-\Cesc M) +\frac{2C_{\Lambda}+2}{M^2},
\end{equation}
and the result follows by noticing that $\Sigma'_{k}$ is distributed according to $\mmu_{M|\tilde{\Lambda}^{(i)}}$ and that by Lemma \eqref{dtvLambdailemma}, 
\[\dtv\left(\mmu_{M|\tilde{\Lambda}^{(i)}}, \mmu_{M}^{(i)}\right)\le \frac{C_{\Lambda}}{M^2}.\qedhere\]
\end{proof}

\section*{Acknowledgements}
AJ \& DP  were supported by AcRF Tier 2 grant R-155-000-143-112. AT and DP were supported by AcRF Tier 1 grant R-155-000-150-133. AJ is affiliated with the Risk  Management Institute and the Centre for Quantitative Finance at the National University of Singapore. We thank Pierre Del Moral for his encouragement and insightful comments. We thank 
Tobias Terzer for noticing several typos and one error in the paper. We thank the anonymous referees for their insightful remarks.

%\begin{supplement}[id=suppA]
%\sname{Supplement A}
%\stitle{Appendix}
%\slink[doi]{COMPLETED BY THE TYPESETTER}
%\sdatatype{.pdf}
%\sdescription{Some technical proofs from the paper are included here.}
%\end{supplement}

\nocite{EberleMartinelliPTRF}
\nocite{GrafakosFourier}
\nocite{Grimmetttext}
\nocite{peresbook}
\nocite{schweizer2012non}
\nocite{Tuckervalidatednumerics}
\nocite{Courseinmodernanalysis}

\bibliographystyle{plainnat}
\bibliography{References}

%\newpage
%    \setcounter{page}{0}
%    \pagenumbering{arabic}
%    \setcounter{page}{1}
\appendix

\section{Appendix}
\subsection{Bounds for the multimodal case with no mixing between modes}\label{SecNoMixing}
In this section, we establish bounds to $V_n(\phi)$ that are applicable to multimodal distributions. Consider a non-trivial partition 
\begin{align} \label{eq.partition}
E = F^{(1)} \sqcup \ldots \sqcup F^{(m)}
\end{align}
of the state space $E$; for simplicity, we assume that $\mmu_k(F^{(r)}) > 0$ for any indexes $1 \leq k \leq n$ and $1 \leq r \leq m$. Importantly, we assume in this section that there is no mixing between the modes in the sense that
\begin{align} \label{nomixingbetweenmodescondeq}
\KK_k(x_r, F^{(r)}) = 1
\end{align}
for any $x_r \in F^{(r)}$ and $1 \leq r \leq m$ and $1\leq k \leq n$; consequently, one can define the restriction $\KK_{k,r}$ of the Markov kernel $\KK_k$ to the mode $F^{(r)}$. In other words, $\KK_{k,r}$ is a Markov kernel on $F^{(r)}$. This very setting is investigated in the articles  \citep{schweizer2012non} and \citep{EberleMartinelliPTRF}. Since $\KK_k$ lets $\mmu_k$ invariant, the restricted Markov kernel $\KK_{k,r}$ lets $\mmu_{k,r}$ invariant where
\begin{align*}
\mmu_{k,r}(S) \defby \frac{\mmu_k(S)}{\mmu_k(F^{(r)})}.
\end{align*}
for a measurable subset $S \subset F^{(r)}$.
Consider the orthogonal projection operator $\wtilde{\mmu}_k:L^2(\mmu_k) \to L^2(\mmu_k)$ defined as
\begin{align*}
\wtilde{\mmu}_k(\phi)
\; = \; 
\sum_{r=1}^m  \mmu_{k,r}(\phi_{| F^{(r)}}) \, 1( \cdot \in F^{(r)}).
\end{align*}
It is important to emphasize that the bounds developed in this section do not exploit global properties of the Markov kernels $\KK_k$. Instead, we leverage a uniform lower bound on the mixing properties of the restricted Markov kernels $\KK_{k,r}$; we set
\begin{align} \label{eq.local.sp.gap}
\gamma_{\KK}^{\mathrm{loc}}
\; = \;
1 - \max\curBK{ \vertiii{\KK_{k,r} - \mmu_{k,r}}_{L^2(\mmu_{k,r})}
\; : \; 
1 \leq k \leq n, \;
1 \leq r \leq m }.
\end{align}
If the Markov kernels $\KK_k$ were reversible, this would translate in a uniform lower bound on the absolute spectral gaps of the restricted Markov kernels $\KK_{k,r}$. Note that, since we assume that there is not mixing between the modes, the spectral gaps of the Markov kernels $\KK_k$ are zero. Furthermore, definition \eqref{eq.local.sp.gap} readily yields that for a test function $\phi \in L^2(\mmu_k)$ we have that
\begin{align*}
\norm{(\KK_k - \widetilde{\mmu}_k) \, \phi}_{L^2(\mmu_k)}^2 
\leq  (1-\gamma_{\KK}^{\mathrm{loc}})^2 \, \norm{ \phi }_{L^2(\mmu_k)}^2.
\end{align*}
%
%
%Similarly to the previous section, we set
%
%\begin{align*}
%\gamma^*_{k,i}\defby 1-\|\mtx{K}_{k,i}-\mtx{\mmu}_{k,i}\|%_{\mmu_{k,i}} \text{ for }1\le k\le n, 1\le i\le m,\text{ and %}\gamma_{\KK}^{\mathrm{loc}} \defby  \min_{1\le k\le n, 1\le i\le %m}\gamma_{k,i}^*.
%\end{align*}
%If $K_{k,i}$ is reversible, then $\gamma^*_{k,i}$ corresponds to the %absolute spectral gap of $K_{k,i}$ (which can be thought as the %''local'' absolute spectral gap of $K_k$).
%
We establish in this section bound on the asymptotic variance $V_n(\phi)$ that involve the \emph{growth-within-mode} constant $A$ of the partition \eqref{eq.partition} defined as
\begin{align} \label{eq.A}
A = \max\curBK{ \frac{\mmu_{k}(F^{(r)})}{\mmu_{j}(F^{(r)})} \; : \; 0 \leq j < k \leq n, \; 1 \leq r \leq m  }.
\end{align}
One can readily see that the growth-within-mode constant $A$ is larger than one. The main result of this section is the following.
%
%We obtain in this section a bound on the asymptotic variance $V_n(\phi)$ in terms of the measure of local mixing $\gamma_{\KK}^{\mathrm{loc}}$ and the \emph{growth-within-mode constant} $A_{j,k}$ defined as
%%
%\begin{align} 
%A_{j,k}\defby  \max_{i\le m} \; \frac{\mmu_{k}(F^{(i)})}{\mmu_j(F^{(i)})}
%\qquad \text{ and } \qquad 
%A\defby \max_{1\le j<k\le n}A_{j,k}.
%\end{align}
%Note that for some of the indices $1\le i\le m$, the ratio $\frac{\mmu_{k}(F^{(i)})}{\mmu_j(F^{(i)})}$ can be less than one, but the maximum of the ratios is at least one nevertheless, hence the name \emph{growth-within-mode}.
%%
%%\todo{$\mmu_k(\phi; 1_{F^{(i)}})$ is a non-standard notation, I prefer to use $\mmu_k(\phi\cdot 1_{F^{(i)}})$ instead.}
%%
%Note that $\wtilde{\mmu}_k$ is self-adjoint projection in $L^2(\mmu_k)$ i.e. $\wtilde{\mmu}_k^2=\wtilde{\mmu}_k$, and for any test function $\phi \in L^2(\mmu_k)$ we have
%%
%\begin{align*}
%\norm{(\KK_k - \widetilde{\mmu}_k) \, \phi}_k^2 
%\leq  (1-\gamma_{\KK}^{\mathrm{loc}})^2 \, \norm{(I_d - \widetilde{\mmu}_k) \, \phi}_k^2
%\leq  (1-\gamma_{\KK}^{\mathrm{loc}})^2 \, \norm{ \phi }_k^2.
%\end{align*}
%%
%
%  MAIN THEOREM of SECTION
%
\begin{theorem}[Variance bound for multimodal case without mixing between modes]\label{thmnomixing}
Assume that there is no mixing between the modes, i.e. condition \eqref{nomixingbetweenmodescondeq} holds, and that 
\begin{align} \label{eq.Gamma.gamma.multimodal}
\Gamma_g \; < \; \frac{1}{ (1-\gamma_{\KK}^{\mathrm{loc}})^2}.
\end{align}
For any test function $\phi \in L^{2+}_0(\mmu)$, the CLT \eqref{eq.CLT.statement} holds with an asymptotic variance $V_n(\phi)$ such that
\begin{align*}
V_n(\phi) 
\; \leq \; 
\BK{ 1+\frac{n \, A \, \Gamma_g}{1-(1-\gamma_{\KK}^{\mathrm{loc}})^2 \cdot \Gamma_g} } \cdot \Var_{\mmu}(\phi).
\end{align*}
\end{theorem}
%
%  PROOF
%
\begin{proof}
The approach is similar to the proof of Theorem \ref{thmuni}. Since $\KK_{k,r}$ lets $\mmu_{k,r}$ invariant, we have that  
$\KK_k \,\wtilde{\mmu}_k = \wtilde{\mmu}_k \, \KK_k = \wtilde{\mmu}_k $; for any test function $\phi \in L^2(\mmu_k)$ we thus have that 
$\norm{\KK_k \phi}_{L^2(\mmu_k)}^2 = \norm{(\KK_k - \wtilde{\mmu}_k) \phi}_{L^2(\mmu_k)}^2 + \norm{\wtilde{\mmu}_k \phi}_{L^2(\mmu_k)}^2$. It has already been established in the proof of Theorem \ref{thmuni} that 
\begin{align*}
V_{k,n}(\phi) \leq \Gamma_g \, \norm{ \KK_{k+1} \, \GG_{k+1,k+2} \cdot\ldots \cdot\GG_{n-1,n}\KK_n \phi}^2_{L^2(\mmu_{k+1})}.
\end{align*}
Furthermore, for $0 \leq k \leq n-1$, we have that
\begin{align*}
&\norm{ \KK_{k+1} \, \GG_{k+1,k+2} \cdot\ldots \cdot\GG_{n-1,n}\KK_n \phi}^2_{L^2(\mmu_{k+1})}
\\
&\leq  
\norm{ (\KK_{k+1} - \wtilde{\mmu}_{k+1}) \, \GG_{k+1,k+2} \, \cdot\ldots \cdot\GG_{n-1,n}\KK_n \phi}^2_{L^2(\mmu_{k+1})}\\
&+ 
\norm{ \wtilde{\mmu}_{k+1} \, \GG_{k+1,k+2} \, \cdot\ldots \cdot\GG_{n-1,n}\KK_n \phi}^2_{L^2(\mmu_{k+1})}\\
&\leq
(1-\gamma_{\KK}^{\mathrm{loc}})^2 \, 
 \norm{\GG_{k+1,k+2} \, \cdot\ldots \cdot\GG_{n-1,n}\KK_n \phi}^2_{L^2(\mmu_{k+1})}\\
 &+
\norm{ \wtilde{\mmu}_{k+1} \, \GG_{k+1,k+2} \, \cdot\ldots \cdot\GG_{n-1,n}\KK_n \phi}^2_{L^2(\mmu_{k+1})}\\
&=
(1-\gamma_{\KK}^{\mathrm{loc}})^2 \, 
V_{k+1,n}(\phi) +
\norm{ \wtilde{\mmu}_{k+1} \, \GG_{k+1,k+2} \, \cdot\ldots \cdot\GG_{n-1,n}\KK_n \phi}^2_{L^2(\mmu_{k+1})}.
\end{align*}
Algebra yields that for a test function $\phi \in L^2(\mmu)$ we have that 
\begin{align*}
\wtilde{\mmu}_{k+1} \GG_{k+1,k+2} \KK_{k+2} \, \cdot\ldots \cdot\GG_{n-1,n}\KK_n \phi
\; = \; 
\sum_{r=1}^m  \curBK{ \frac{\mmu_{n}(F^{(r)})}{\mmu_{k+1}(F^{(r)})} } \, \wtilde{\mmu}_{n} \phi.
\end{align*}
By definition \eqref{eq.A} of the growth-within-mode constant $A \geq 1$ and the fact that $\mmu(\phi) = 0$, it follows that
\begin{align*}
&\norm{ \wtilde{\mmu}_{k+1} \, \GG_{k+1,k+2} \, \cdot\ldots \cdot\GG_{n-1,n}\KK_n \phi}^2_{L^2(\mmu_{k+1})}
=
\sum_{r=1}^m \mmu_{k+1}(F^{(r)}) \, \curBK{\frac{\mmu_{n}(F^{(r)})}{\mmu_{k+1}(F^{(r)})}}^2 \, \mmu_{n,r}(\phi)^2\\
&\quad =
\sum_{r=1}^m \curBK{\frac{\mmu_{n}(F^{(r)})}{\mmu_{k+1}(F^{(r)})}} \, \mmu_{n}(\phi_{|F^{(r)}})^2
\leq
A \times \sum_{r=1}^m \mmu_{n}(\phi_{|F^{(r)}})^2
\leq 
A \times  \Var_{\mmu}(\phi).
\end{align*}
Consequently, 
$V_{k,n}(\phi) \leq \Gamma_g \, (1-\gamma_{\KK}^{\mathrm{loc}})^2 \, V_{k+1,n}(\phi) + \Gamma_{g} \, A \, \Var_{\mmu}(\phi)$.
Since $V_{n,n}(\phi) = \Var_{\mmu}(\phi)$ $\leq \Gamma_{g} \, A \, \Var_{\mmu}(\phi)$, iterating this bound and exploiting the fact that $\Gamma_{g} \, (1-\gamma_{\KK}^{\mathrm{loc}})^2 < 1$ yields that
\begin{align*}
V_{k,n}(\phi) \leq \frac{\Gamma_{g} A}{1-\Gamma_{g} \, (1-\gamma_{\KK}^{\mathrm{loc}})^2} \Var_{\mmu}(\phi).
\end{align*}
Since $V_n(\phi) = \sum_{k=0}^{n} V_{k,n}(\phi)$, the conclusion follows.
\end{proof}

Theorem \ref{thmnomixing} gives an improvement over the quadratic  bound provided by Theorem $1.4$ of \cite{schweizer2012non}. Note that the variance bound in Theorem \ref{thmnomixing} grows linearly with  the number of resampling stages $n \geq 1$, in contrast with Theorem \ref{thmuni} where the variance was bounded independently of $n$. This is caused by the variance in the ratios of particles in different modes, introduced by the repeated resampling steps. In general, this dependence on $n \geq 1$ cannot be removed as can be seen by considering a discrete state space $E$ with only two elements, a partition with two modes, and Markov kernels with no mixing between the modes. Theorem \ref{thmwithmixing1} is rather different from Theorem \ref{thmnomixing} where no mixing was allowed between the modes. It might be tempting to conjecture that the results of Theorem \ref{thmnomixing} should hold even when there is mixing between the modes since, at an heuristic level, one may think that mixing between the modes can only help the Markov chain to reach its stationary distribution more efficiently, and thus it should decrease the asymptotic variance of the SMC sampler. We have found out that this is not always the case. In the following counterexample, the state space $E=\{1,2,3,4\}$ consists of only $4$ elements. The SMC algorithm consists of two stages, $\mmu_0$ is uniform on $E$, and
$\mmu_1\defby\mmu$ takes the elements of $1,2,3,4$ with probabilities $0.1319, 0.1778, 0.0638,    0.6265$. The reversible transition kernel $\KK_1$ is given by
\begin{align*}
\KK_1\defby\left(\begin{matrix}
0.5520   & 0.1858   & 0.0413   & 0.2209\\
0.1378  &   0.7837 &   0.0769 &   0.0016\\
0.0853   & 0.2145  &  0.6311  &  0.0691\\
0.0465    &0.0004    & 0.0070  &  0.9460
\end{matrix}\right).
\end{align*}
The modes are chosen as $F^{(1)}\defby\{1,2\}$ and $F^{(2)}\defby\{3,4\}$. Then the version of $\KK_1$ that does not allow for mixing between the modes is setting the transition probabilities between these two sets to $0$, and changing the probabilities of staying in place by the corresponding amount. This transition matrix is denoted by $\KK_1^{\rm{nomix}}$, and in our case, it equals
\begin{align*}
\KK_1^{\rm{nomix}}\defby\left(\begin{matrix}
 0.8142   & 0.1858   & 0   & 0\\
0.1378  &   0.8622 &   0 &   0\\
0   & 0  &  0.9309  &  0.0691\\
0    &0    & 0.0070  &  0.9930
\end{matrix}\right).
\end{align*}
The function $\phi$ is chosen to take values $0.3973,-0.5697,-0.3222, 0.1109$ on $1,2,3,4$, respectively. Based on \eqref{eq.asymp.var.expansion}, the asymptotic variance for the first case (with mixing) equals $V_1(\phi)=0.1669$, while in the second case (without mixing) it equals $V_1^{\rm{nomix}}(\phi)=0.1579$. So despite the fact that $\KK_1$ has better global mixing properties than $\KK_1^{\rm{nomix}}$, the SMC algorithm based on $\KK_1$ still has bigger asymptotic variance than the one based on $\KK_1^{\rm{nomix}}$.

\subsection{Proof of Lemma \ref{worstdriftlemma}}
For $1\le i,j\le 3, i\ne j$ define $s^{i\to j}:=s+\frac{e_j}{M}-\frac{e_i}{M}$ and let $P_{i\to j}(s_1,s_2, s_3)$ denote the probability a step in the magnetisation chain started at $(s_1,s_2, s_3)$ changes a spin of colour $i$ to colour $j$. Let $P_{\circlearrowleft}(s_1,s_2, s_3)$ denote the probability of staying in place. Let $k$ be the colour in $\{1,2,3\}$ that is different from $i$ and $j$. Then it is straightforward to show that
\begin{align*}P_{i\to j}(s_1,s_2, s_3)&=
\frac{s_i\exp\left[\frac{2}{M}+2\tbeta_c(s_j-s_i)\right]}{1+\exp\left[\frac{2}{M}+2\tbeta_c(s_j-s_i)\right]+\exp\left[\frac{2}{M}+2\tbeta_c(s_k-s_i)\right]}\\
&=\frac{s_i\exp(2\tbeta_c s_j)}{\exp(-\frac{2}{M}+2\tbeta_c s_i)+\exp(2\tbeta_c s_j)+\exp(2\tbeta_c s_k)}.
\end{align*}
Using these notations, we have
\begin{align*}
\E(d_C(S(1))-d_C(S(0))|S(0)=s)= \sum_{1\le i,j\le 3, i\ne j} P_{i\to j}(s_1,s_2,s_3)\left(d_C\left(s^{i\to j}\right)-d_C(s)\right).
\end{align*}
Figure \ref{fig4modes} illustrates the position of the modes and the change of distance from the one of the centres from $r:=d_C(s)$ to $r':=d_C\left(s^{i\to j}\right)$ by moving along direction $i\to j$ with distance $h:=\frac{1}{M}$. By Pythagoras' theorem, we can see that $(r')^2=(h\sin(\beta))^2 + (r-h\cos(\beta))^2$, which implies that $|r'-(r-h\cos(\beta))|\le \frac{h^2}{2(r-h)}$ for $r>h$.

Suppose that $s\in T_I$ for some $1\le I\le 4$. Let 
\[\ol{s}=(\ol{s}_1,\ol{s}_2,\ol{s}_3):=C_I-s=(C_{I,1}-s_1,C_{I,2}-s_2,C_{I,3}-s_3)\]
be the vector from $s$ to $C_I$. Using the standard 2 dimensional Euclidean scalar product between the vector corresponding these barycentric coordinates and $e_j-e_i$, we have
\begin{align*}h\cos(\beta)&=\frac{1}{M}\frac{\left<\ol{s}_1e_1+\ol{s}_2 e_2 +\ol{s}_3e_3,e_j-e_i\right>}{d_C(s)}=\frac{1}{2 M}\frac{\ol{s}_j-\ol{s}_i}{d_C(s)},
\end{align*}
where we have used the fact that $\left<e_l, e_m\right>=\frac{1}{3}\cdot 1_{[l=m]}- \frac{1}{6}\cdot 1_{[l\neq m]}$. Moreover, based on the definition of the Glauber dynamics, one can show that 
\begin{align*}&\left|P_{i\to j}(s_1,s_2, s_3)-\frac{s_i\exp\left[2\tbeta_c s_j\right]}{\exp\left[2\tbeta_c s_1 \right]+\exp\left[2\tbeta_c s_2\right]+\exp\left[2\tbeta_c s_3 \right]}\right|\le \exp\left(\frac{2}{M}\right)-1\le \frac{8}{M}.
\end{align*}
By combining these facts, it follows that for any $s\in \Omega^{S}$ with $d_C(s)>\frac{1}{M}$, we have
\begin{align}
\nonumber&\E(d_C(S^s)-d_C(s))-\left(\frac{1}{2 M}\frac{1}{d_C(s)}\sum_{1\le i,j\le 3, i\ne j} \frac{(\ol{s}_j-\ol{s}_i)s_i \exp\left[2\tbeta_c s_j\right]}{\exp\left[2\tbeta_c s_1\right]+\exp\left[2\tbeta_c s_2\right]+\exp\left[2\tbeta_c s_3\right]}\right)
\\
\label{driftupperbndeq1}&\le \frac{1}{M^2}\cdot \left(8+\frac{1}{2(d_C(s)-1/M)}\right).
\end{align}
Note that although  $s$ and $s^{i\to j}$ can be in different triangles, the fact that $d_C(s^{i\to j})\le d_C(s^{i\to j},C_I)$ guarantees that \eqref{driftupperbndeq1} is still valid in such cases. Let us denote 
\begin{equation}
L(s):=-\frac{1}{2d_C(s)}\sum_{1\le i,j\le 3, i\ne j} \frac{(\ol{s}_j-\ol{s}_i)s_i \exp\left[2\tbeta_c s_j\right]}{\exp\left[2\tbeta_c s_1\right]+\exp\left[2\tbeta_c s_2\right]+\exp\left[2\tbeta_c s_3\right]},
\end{equation}
then the statement of our lemma would follow from the inequality $L(s)\le -\varphi(d_C(s))$.
Figure \ref{FigLsvarphisineq} illustrates numerically this inequality (the small circles correspond to the values of $L(s)$ in function of $d_C(s)$ for points $s$ chosen on a fine triangular grid on $T$, while the continuous curve is $-\varphi(d_C(s))$ as a function of $d_C(s)$). 

\begin{figure}
\begin{center}
\includegraphics[height=4cm]{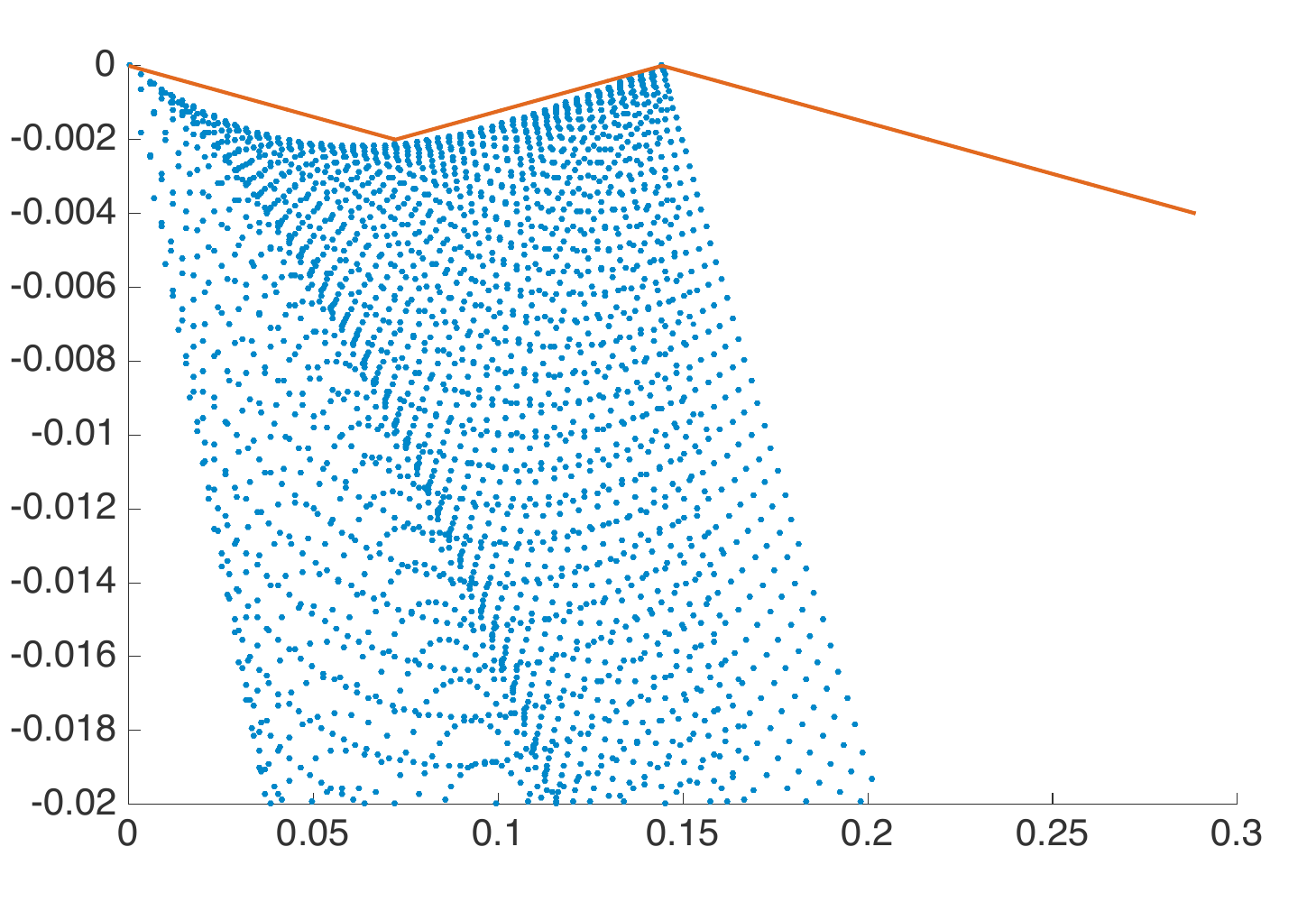}
\end{center}
\caption{Illustration of the inequality $L(s)\le -\varphi(d_C(s))$.}
\label{FigLsvarphisineq}
\end{figure}
This inequality can be proven as follows. First, it is not difficult to show that both $L(s)$ and $\varphi(d_C(s))$ are Lipschitz with respect to the $d_C(s)$ distance, with some finite constants $C_L$ and $C_{\phi}$. Now if we take a sufficiently fine triangular grid on $T$, then one can rigorously numerically check the inequality using interval arithmetics at each grid-point, and by the Lipschitz condition, a small neighbourhood around them (see the monograph \cite{Tuckervalidatednumerics} for an introduction to interval arithmetics). Based on such a numerical check, one can see that only regions where the inequality is not yet proven are small neighbourhoods around $C_1$, $C_2$, $C_3$ and $C_4$, and the half points $C_{14}:=(C_1+C_4)/2$, $C_{24}:=(C_2+C_4)/2$,
$C_{34}:=(C_3+C_4)/2$. At all of these points, $L(s)=-\varphi(d_C(s))=0$. Using a Taylor expansion with remainder term at these points, the inequality can be extended to these neighbourhoods as well, completing the proof. 

\subsection{Proofs of preliminary lemmas}\label{Secprelimlemmaproof}
\begin{proof}[Proof of Lemma \ref{lemmaanticoncentration}]
For $k\in \Z_+$, let $D_{k}:=X_{k}-X_{k-1}$,  and $\ol{D}_k:=D_{k}-\E(D_{k}|\F_{k-1})$. For $l\in \N$, let
\begin{equation}\label{olXldef}
\ol{X}_l:=x_0+\sum_{k=1}^{l}\ol{D}_k.
\end{equation}
It is clear that $(\ol{X}_{k})_{k\ge 0}$ is a martingale. Using assumptions (1) and (2), we have $\E(D_{k}|\F_{k-1})$ $\le \delta\le R$, and thus $\ol{X}_t\ge X_t-\delta t$ and $|\ol{X}_{k+1} - \ol{X}_{k}|\le 2R$ almost surely for every $k\ge 0$. From assumption (4), it follows that $\E(\ol{D}_k^2|\F_{k-1})\ge v$ for every $k\ge 1$. 

Let $\tau:=\inf\{k\ge 0: |\ol{X}_k|>z\sqrt{v}\}$ be the exit time of $(\ol{X}_t)_{t\ge 0}$ from the set $[-z\sqrt{v},z\sqrt{v}]$. 
Using the fact that the increments are bounded by $2R$ in absolute value, we have that $|\ol{X}_{\min(\tau, t)}|\le z\sqrt{v}+2R$. Moreover, it is easy to see that $\E(\ol{X}_k^2-\ol{X}_{k-1}^2|\F_{k-1})\ge v$, so $Y_k:=\ol{X}_k^2-kv$ and $Y_{\min(\tau, k)}$ are submartingales, in particular,
\[\E(Y_{\min(\tau, k)})=\E(\ol{X}_{\min(\tau, t)}^2-\min(\tau,k)v)\ge 0.\]
By the bounded convergence theorem, it follows that 
$\lim_{k\to \infty}\E(\ol{X}_{\min(\tau, t)}^2)=\E(\ol{X}_{\tau}^2)\le (z\sqrt{v}+2R)^2$,
and by the monotone convergence theorem, we have 
$\lim_{k\to \infty} \E(\min(\tau,k)v)=\E(\tau)v$, so we have $\E(\tau)\le \frac{(z\sqrt{v}+2R)^2}{v}$, and by Markov's inequality, this implies that
\begin{equation}\label{exitlemmaeq1}
\PP\left(\tau > 4(z+2R/\sqrt{v})^2\right)\le \frac{1}{4}.
\end{equation}
By the above bound on $\E(\tau)$, it follows that $\PP(\tau<\infty)=1$. By applying the optional stopping theorem (see Section 12.5 of \cite{Grimmetttext}) to the martingale $(\ol{X}_k)_{k\ge 0}$, it follows that $\E(\ol{X}_\tau)=0$, and since $\ol{X}_\tau\in [-z\sqrt{v}-2R, -z\sqrt{v})\cup (z\sqrt{v}, z\sqrt{v}+2R]$, by writing 
\[0=\E(\tau)=\E\left( \II[ \ol{X}_\tau>0]\cdot (z\sqrt{v}+(\ol{X}_\tau-z\sqrt{v}))+ \II[ \ol{X}_\tau<0]\cdot (-z\sqrt{v}+(\ol{X}_\tau+z\sqrt{v}))\right),\]
one can show that 
\begin{equation}\label{exitlemmaeq2}
\PP(\ol{X}_{\tau}<-z\sqrt{v})\ge \frac{1}{2}-\frac{R}{z\sqrt{v}}.
\end{equation}
Using \eqref{exitlemmaeq1}, \eqref{exitlemmaeq2} and the condition $z>12\frac{R}{\sqrt{v}}$, by the union bound, it follows that 
\[\PP\left(\ol{X}_{\tau}<-z\sqrt{v} \text{ and } \tau \le 4(z+2R/\sqrt{v})^2\right)\ge \frac{1}{2}-\frac{1}{12}-\frac{1}{4}=\frac{1}{6},\]
and the claim of the theorem now follows by rearrangement using condition (2).
\end{proof}

\begin{proof}[Proof of Lemma \ref{lemmamoving}]
\eqref{eqmovingbackward} is a direct consequence of the Azuma-Hoeffding inequality applied to $\ol{X}_k$ (as defined in \eqref{olXldef} of the previous lemma). For \eqref{eqmovingforward}, we only need to notice that by the Azuma-Hoeffding inequality applied to $\ol{X}_k$, 
\[\PP(X_k>T)\le \PP(\ol{X}_k>T+\delta k)\le \exp\left(-\frac{(T+\delta k)^2}{2 k R^2}\right)\le \exp\left(-\frac{T\delta}{R^2}\right),\]
and the result follows by summing over $1\le k\le l$.
\end{proof}

\begin{proof}[Proof of Lemma \ref{lemmacinmagtocinspinsPotts}]
The proof of this lemma is similar to the proof of Lemma 2.9 of \cite{LevinLuczakPeres}.
We are going to describe the coupling recursively. First, note that $\sigma(0)$ and $\tsigma(0)$ are constants. For $t\ge 0$, suppose that $((\sigma(k))_{0\le k\le t},(\tsigma(k))_{0\le k\le t})$ is already defined, and let $I_t$ be chosen uniformly from $[M]:=\{1,2,\ldots, M\}$. Based on the definition of the Glauber dynamics, it is easy to show that the probability that in the next step we replace $\sigma_{I_t}(t)=i$ by a different colour $j$ or $k$ is given by 
\begin{align*}p_j&:=\frac{\exp[2\tbeta_c s_j(\sigma(t))]}{\exp[2\tbeta_c s_j(\sigma(t))]+\exp[2\tbeta_c s_k(\sigma(t))]+\exp[2\tbeta_c s_i(\sigma(t))-2/M]}\\
p_k&:=\frac{\exp[2\tbeta_c s_k(\sigma(t))]}{\exp[2\tbeta_c s_j(\sigma(t))]+\exp[2\tbeta_c s_k(\sigma(t))]+\exp[2\tbeta_c s_i(\sigma(t))-2/M]},
\end{align*}
and the probability of staying in place is given by
\[p_i:=\frac{\exp[2\tbeta_c s_i(\sigma(t))-2/M]}{\exp[2\tbeta_c s_j(\sigma(t))]+\exp[2\tbeta_c s_k(\sigma(t))]+\exp[2\tbeta_c s_i(X(t))-2/M]}.\]
Let $Z_{t+1}$ be a random variable taking values in the set $\{1,2,3\}$ according to these probabilities (independently of the previously defined random variables).

We let $\sigma_{I_t}(t+1):=Z_{t+1}$, and set the other spins to the same as in $\sigma(t)$. If $\sigma_{I_t}(t)=\tsigma_{I_t}(t)$, then we let $\tsigma_{I_t}(t+1):=Z_{t+1}$, and keep the rest of its spins the same as in $\tilde{X}(t)$. If $\sigma_{I_t}(t)\ne \tsigma_{I_t}(t)$, then we pick $\tilde{I}_t$ uniformly from 
\[\{l\in [M]: \tsigma_{t}(l)=\sigma_{I_t}(t), \tsigma_l(t)\ne \sigma_l(t)\},\]
set $\tsigma_{\tilde{I}_t}(t+1):=Z_{t+1}$, and keep the other spins in $\tsigma(t+1)$ the same as in $\tsigma(t)$.

By this choice, we guarantee that $s(\sigma(t))=s(\tsigma(t))$ for $t\ge 0$. Let us denote the Hamming distance of $\sigma(t)$ and $\tsigma(t)$ by $D(t):=\sum_{l=1}^{M}1_{[\sigma_l(t)\ne \tsigma_l(t)]}$. Based on the definition of this coupling, we have that $D(t+1)\le D(t)$. Moreover, if $\sigma_{I_t}(t)\ne \tsigma_{I_t}(t)$, but $Z_{t+1}=\sigma_{\tilde{I}_t}(t)$ or $Z_{t+1}=\tsigma_{I_t}(t)$, then $D(t+1)$ decreases by one compared to $D(t)$. The probability of each of these two events is at least
$\frac{D(t)}{M}\cdot \frac{1}{2+\exp(2\tbeta_c)}=\frac{D(t)}{18M}$, therefore
\[\E(D(t+1)-D(t)|\sigma(t),\tsigma(t))\le -\frac{D(t)}{9 M}.\]
For any $t\in \N$, let $Y(t):=D(t)\cdot \left(1-\frac{1}{9M}\right)^{-t}$. Then it is easy to show that this is a non-negative supermartingale, and therefore
\[\E(D(t))\le \E(D(0))\cdot \left(1-\frac{1}{9M}\right)^{-t}\le M\left(1-\frac{1}{9M}\right)^{t}.\]
The result now follows from the fact that $\E(D(t))\ge 2 \PP(\tau>t)$.
\end{proof}

\subsection{Proof of curvature bound}\label{secproofcurvaturebound}
\begin{proof}[Proof of Proposition \ref{propcurvature}]
Consider bounding the curvature $\kappa(s,s^{i\to j})$, where $s^{i\to j}:=s+\frac{e_j}{M}-\frac{e_i}{M}$, for some $i\ne j$ in $\{1,2,3\}$. By the definition of $\kappa$, we have
\[\kappa(s,s^{i\to j})=1-W_1(\P^{\mathrm{mag}}_{(m)}(s,\cdot), \P^{\mathrm{mag}}_{(m)}(s^{i\to j}, \cdot)).\]
To compute this expression, we note that there are 7 possible moves. We stay in place with probability $P_{\circlearrowleft}(s)$, or move to $s^{k\to l}$ with probability $P_{k\to l}(s)$.
To understand the change in the distributions $\P^{\mathrm{mag}}_{(m)}(s,\cdot)$ and $\P^{\mathrm{mag}}_{(m)}(s^{i\to j},\cdot)$, we are going to look at the partial derivatives of the transition probabilities. By Taylor's theorem with the remainder term, we have
\begin{align}\label{eqPklbound}
&\left|\left(P_{k\to l}(s^{i\to j})-P_{k\to l}(s)\right)-\frac{1}{M}\left(\frac{\partial P_{k\to l}(s)}{\partial s_j}-\frac{\partial P_{k\to l}(s)}{\partial s_i}\right) \right|
\\
\nonumber&\le \frac{1}{M^2}\sup_{s\in \Lambda^{(m)}}\left(
\left|\frac{\partial^2 P_{k\to l}(s)}{\partial^2 s_i}\right| +\left|\frac{\partial^2 P_{k\to l}(s)}{\partial^2 s_j}\right| +\left|\frac{\partial^2 P_{k\to l}(s)}{\partial s_i\partial s_j}\right| \right).
\end{align} 
It is straightforward to show that 
\begin{equation}\label{eqsecderbound}
\left|\frac{\partial^2 P_{k\to l}(s)}{\partial s_i\partial s_j}\right|\le 24 \text{ and }\left|\frac{\partial^2 P_{\circlearrowleft}(s)}{\partial s_i\partial s_j}\right|\le 144\end{equation}
for any $1\le i, j\le 3$, $1\le k, l \le 3, k\ne l$.
Moreover, if we define the $M$-free versions of $P_{k\to l}(s)$ as 
\[\tilde{P}_{k\to l}(s_1,s_2, s_3):=\frac{s_k\exp(2\tbeta_c s_l)}{\exp(2\tbeta_c s_1)+\exp(2\tbeta_c s_2)+\exp(2\tbeta_c s_3)},\]
then one can show that
\begin{equation}\label{eqPtildabound}\left|\frac{\partial P_{k\to l}(s)}{\partial s_j}-\frac{\partial \tilde{P}_{k\to l}(s)}{\partial s_j}\right|\le \frac{12}{M}.\end{equation}
From equations \eqref{eqPklbound} and \eqref{eqPtildabound}, we can see that
\begin{equation}
P_{k\to l}(s^{i\to j})-P_{k\to l}(s)=\frac{1}{M}\left(\frac{\partial \tilde{P}_{k\to l}(s)}{\partial s_j}-\frac{\partial \tilde{P}_{k\to l}(s)}{\partial s_i}\right)\cdot \left(1+\OO\left(\frac{1}{M}\right)\right).
\end{equation}
Thus asymptotically the difference in the distributions $\P^{\mathrm{mag}}_{(m)}(s,\cdot)$ and $\P^{\mathrm{mag}}_{(m)}(s^{i\to j},\cdot)$ is entirely determined by the values  $\frac{\partial \tilde{P}_{k\to l}(s)}{\partial s_j}-\frac{\partial \tilde{P}_{k\to l}(s)}{\partial s_i}$. Figure \ref{figcurvaturecoupling} illustrates the values of 
$\frac{\partial \tilde{P}_{k\to l}(s)}{\partial s_j}-\frac{\partial \tilde{P}_{k\to l}(s)}{\partial s_i}$ for 3 cases, 
\ifdefined \aop
\mbox{$s=C_1$, \hspace{1.5mm} $(i\to j)=(1\to 2)$;} \mbox{$s=C_1$, $(i\to j)=(2\to 3)$;} \mbox{$s=C_4$, $(i\to j)=(1\to 2)$.} 
\else
\[s=C_1, (i\to j)=(1\to 2);\quad s=C_1, (i\to j)=(2\to 3);\quad s=C_4, (i\to j)=(1\to 2).\]
\fi
The other centers and directions can be shown to be equivalent to one of these cases because of the symmetry of the problem.

We construct the coupling between the distributions $\P^{\mathrm{mag}}_{(m)}(s,\cdot)$ and $\P^{\mathrm{mag}}_{(m)}(s^{i\to j},\cdot)$ as follows. Firstly, for every $z\in \Lambda^{(m)}$, we couple $z$ with $z^{i\to j}$ with probability
\[\min(\P^{\mathrm{mag}}_{(m)}(s,z),\P^{\mathrm{mag}}_{(m)}(s^{i\to j},z^{i\to j})).\] This way we are left with probabilities of $\OO\left(\frac{1}{M}\right)$ from both distributions, determined by the coefficients $\frac{\partial \tilde{P}_{k\to l}(s)}{\partial s_j}-\frac{\partial \tilde{P}_{k\to l}(s)}{\partial s_i}$. These are illustrated on Figures \ref{curvfig2c}, \ref{curvfig3c} and \ref{curvfig1c}, with the numbers with black colour corresponds to the remaining probabilities of the distribution $\P^{\mathrm{mag}}_{(m)}(s,\cdot)$, while the red colour ones correspond to the remaining probabilities of the distribution $\P^{\mathrm{mag}}_{(m)}(s^{i\to j},\cdot)$. These are then matched together as illustrated by the arrows connecting them. Based on this matching, it is easy to see that for $s\in\{C_1,C_2,C_3,C_4\}\cap \Omega^{S}$, $\kappa(s,s^{i\to j})\ge \frac{0.025}{M}(1+\OO\left(\frac{1}{M}\right))$, and thus $\kappa(s,s^{i\to j})\ge \frac{0.02}{M}$ for $M$ larger than some absolute constant (it might happen that  $C_i\notin \Omega^{S}$ for some $1\le i\le 4$, but this is not an issue due to the continuity argument that follows). To extend this result to the neighbourhood $\Lambda^{(m)}$, we note that using \eqref{eqsecderbound}, one can show that the coefficients $\frac{\partial \tilde{P}_{k\to l}(s)}{\partial s_j}-\frac{\partial \tilde{P}_{k\to l}(s)}{\partial s_i}$ and $\frac{\partial \tilde{P}_{\circlearrowleft}(s)}{\partial s_j}-\frac{\partial \tilde{P}_{\circlearrowleft}(s)}{\partial s_i}$ cannot change by more than $48\cdot 4\rho$, and $288\cdot 4\rho$, respectively. Since the largest distance among the possible steps from $s$ and $s^{i\to j}$ is 3, this implies that the curvature  satisfies that for every $s\in \Lambda^{(m)}$,
\begin{align*}
\kappa(s,s_{i \to j})&\ge \frac{0.02}{M}-\frac{3(6\cdot 48\cdot 4\rho + 288\cdot 4\rho)}{M}\left(1+\OO\left(\frac{1}{M}\right)\right)\\
&=\frac{0.02}{M}-\frac{6912 \rho}{M}\left(1+\OO\left(\frac{1}{M}\right)\right),
\end{align*}
so with the choice $\rho=10^{-6}$ we have made, it follows that $\kappa(s,s_{i \to j})\ge \frac{0.01}{M}$ for $M$ larger than some absolute constant.

Finally, one also needs to bound $\kappa(s,s^{i\to j})$ at boundary of $\Lambda^{(m)}$. This can be done in the same way as the above three cases (the details are omitted due to space considerations), and we obtain that $\kappa(s,s_{i \to j})\ge \frac{0.01}{M}$ for $M$ larger than some absolute constant at every $s, s_{i \to j}\in \Lambda^{(m)}$. Therefore the advertised result holds.
\begin{figure}
  \centering
  \subcaptionbox{$\left(\frac{2}{3},\frac{1}{6},\frac{1}{6}\right)$ direction $1\to 2$ \label{curvfig2}}{\includegraphics[height=3.5cm, bb=0 -1 246 154]{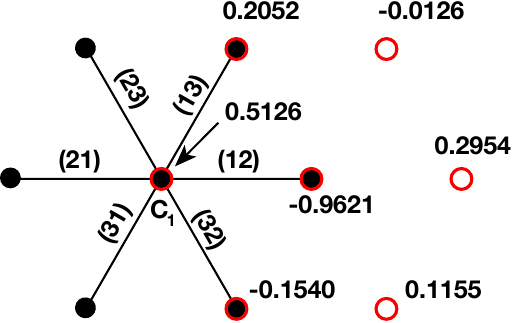}}\hspace{2em}%
  \subcaptionbox{$\left(\frac{2}{3},\frac{1}{6},\frac{1}{6}\right)$ direction $1\to 2$ \label{curvfig2c}}{\includegraphics[height=4cm, bb=0 -1 246 189]{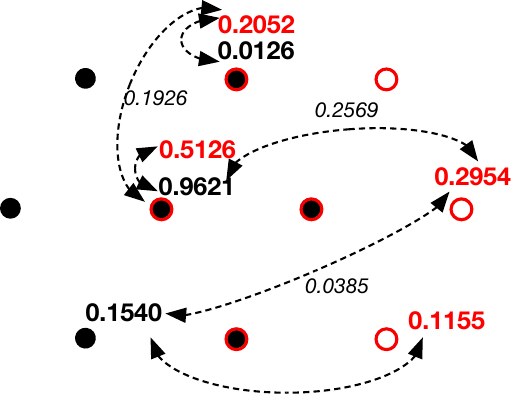}}\\ \vspace{2mm}%
  \subcaptionbox{$\left(\frac{2}{3},\frac{1}{6},\frac{1}{6}\right)$ direction $2\to 3$ \label{curvfig3}}{\includegraphics[height=5.5cm, bb=0 -1 198 212]{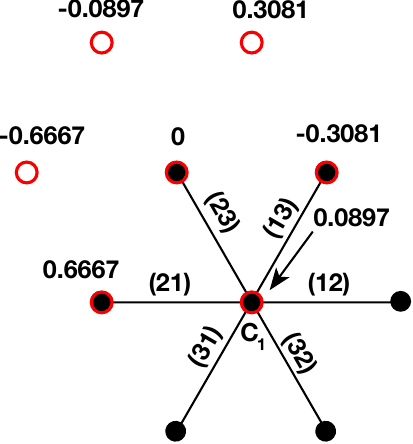}}\hspace{2em}%
  \subcaptionbox{$\left(\frac{2}{3},\frac{1}{6},\frac{1}{6}\right)$ direction $2\to 3$ \label{curvfig3c}}{\includegraphics[height=5.5cm, bb=0 -1 221 217]{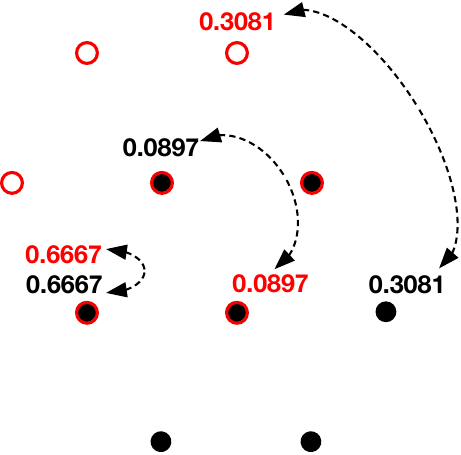}}\\ \vspace{2mm}%
  \subcaptionbox{$\left(\frac{1}{3},\frac{1}{3},\frac{1}{3}\right)$ direction $1\to 2$ \label{curvfig1}}{\includegraphics[height=3.5cm, bb=0 -1 254 152]{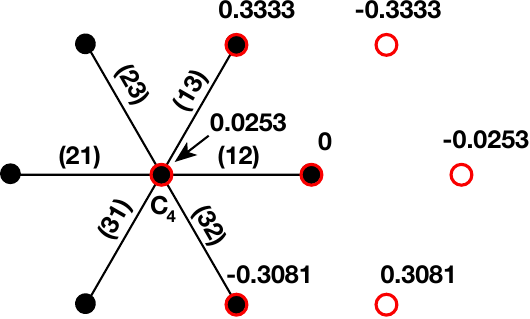}}\hspace{2em}%
  \subcaptionbox{$\left(\frac{1}{3},\frac{1}{3},\frac{1}{3}\right)$ direction $1\to 2$ \label{curvfig1c}}{\includegraphics[height=3.5cm, bb=0 -1 228 165]{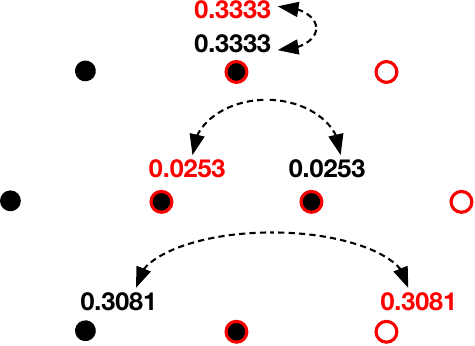}}%
  \caption{Changes in distribution when moving from a central point}
  \label{figcurvaturecoupling}
\end{figure}
\end{proof}

\subsection{Proof of Proposition \ref{keyprop2}}\label{secproofpropgrowthwithinmodePotts}
The proof is based on two preliminary lemmas. The first one analyses the asymptotic version of the log-likelihood of the ratio of spins. As previously, we let $T:=\{s\in [0,1]^3: s_1+s_2+s_3=1\}$.
\begin{lemma}[Maximum of the asymptotic log-likelihood]\label{lemmaL}
Let $\LL : T\to\R$ be defined as 
\begin{equation}\label{eqLdef}
\LL(s):=\sum_{i\le 3}(\tbeta_c s_i^2 - s_i\log(s_i)).
\end{equation}
Then the function $\LL$ takes its maximum on $T$ on the set $\{C_1,\ldots,C_4\}$, and it satisfies that for any $s\in T$,
\[\LL(s)\le \LL(C_1)-c d_C(s)^2, \]
where $d_C(s)$ is defined as in \eqref{dCsdefeq}, and $c>0$ is an absolute constant.
\end{lemma}
\begin{proof}
Assume that $C_I$ is the closest point to $s$ in $\{C_1,\ldots,C_4\}$. By changing to polar coordinates, we can rewrite $s$ as
\[s=C_I-r\cdot \sqrt{\frac{2}{3}}\left(\cos\left(\alpha-\frac{\pi}{6}\right),\cos\left(\alpha-\frac{5\pi}{6}\right),\cos\left(\alpha-\frac{3\pi}{2}\right)\right),\]
where $r=d_C(s)$, and $\alpha$ is the angle of $(C_i,s)$ with the horizontal vector $((1,0,0),(0,1,0))$ (in counterclockwise direction). By writing $\LL$ as a function of $r$ and $\alpha$, one can show that for $0\le r\le \frac{1}{2\sqrt{6}}$, the maximum is taken at $\alpha=\frac{\pi}{6}$ for $C_1$, $\alpha=\frac{5\pi}{6}$ for $C_2$, $\alpha=\frac{3\pi}{2}$ for $C_3$, and 
$\alpha\in \{\frac{\pi}{2},\frac{7\pi}{6},\frac{11\pi}{6}\}$ for $C_4$. Moreover, it can also be shown that for $\frac{1}{2\sqrt{6}}< r\le \frac{1}{\sqrt{6}}$, the maximum in $\alpha$ is taken at a point falling on the edges of the triangle $(0,\frac{1}{2},\frac{1}{2}), (\frac{1}{2},0,\frac{1}{2}), (\frac{1}{2},\frac{1}{2}, 0)$. The claim of the lemma is now easy to verify by substituting these values of $\alpha$ into $\LL$.
\end{proof}

The following lemma shows an  error bound on the convergence of the Riemann-sums of the integral $Z_m:=\int_{t=-\infty}^{\infty}\exp(-t^2) t^{m} \mathrm{d}t$. The bound is somewhat surprising since it implies faster than polynomial convergence. This is due to cancellations in the sum of differences between the integral and the Riemann-sum.
\begin{lemma}[Convergence of Riemman-sum to Gaussian integral]\label{lemmaRiemannapprox}
For $R>0, \delta>0$, let 
\begin{equation}\label{eqPsimrdeltadef}\Psi_{m}(R,\delta):=\sum_{k\in \Z}\exp(-(k+\delta)^2 R^2) (k+\delta)^m R^{m+1}.
\end{equation}
Then 
\begin{equation}\label{eqPsilimit}
\lim_{R\searrow 0 }\Psi_{m}(R,\delta)=Z_m.\end{equation}
Moreover, for any $l, m\in \N$, there exists some universal constant $c_{m,l}<\infty$ such that for every $0<R\le 1$,
\begin{equation}\label{eqPsibound}
\left|\Psi_{m}(R,\delta)-Z_m\right|\le c_{m,l} R^l.
\end{equation}
\end{lemma}
\begin{proof}
The Poisson summation formula (see \cite{GrafakosFourier}) yields that 
\[\Psi_{m}(R,\delta)=
R^{m+1} \, \sum_{k \in \Z} \widehat{G}_{R,\delta}(k),
\]
where $\widehat{G}_{R,\delta}$ is the Fourier transform of 
$G_{R,\delta}(x) \defby \exp\BK{ -(x+\delta)^2 \, R^2 }  \, (x+\delta)^m$. Algebra readily yields that
\begin{align*}
\widehat{G}_{R,\delta}(\xi)
&= 
\int_{\R} G_{R,\delta}(x) \, e^{2i \, \pi x \, \, \xi}
\\
&=
\sqrt{\pi / R^2} \,
e^{2 i \pi \, \delta \, \xi} \, 
e^{-\pi^2 \, \xi^2 / R^2} \, 
\E{ \BK{ Z / \sqrt{2R^2} + i \pi \, \xi / R^2 }^m },
\end{align*}
where $Z \sim \mathcal{N}(0,1)$ is a standard Gaussian random variable.
Let $\xi\in \Z\setminus \{0\}$, then using the fact that $e^{-(\pi^2-\pi) \, \xi^2 / R^2}\cdot R^{-m}$ can be bounded uniformly for $R>0$ by a constant only depending on $m$, we have
\begin{align*}
R^{m+1}\left| \widehat{G}_{R,\delta}(\xi) \right|
\leq
C(m) \, e^{-\pi \, \xi^2 / R^2 } \, (1+|\xi|^m)
\end{align*}
where $C(m)$ is a constant that only depends on the exponent $m$. Since we know that
$R^{m+1} \, \widehat{G}_{R,\delta}(0) = Z_m$, it follows that
\begin{align*}
\left| \Psi_{m}(R,\delta)-Z_m \right|
&= R^{m+1} \, \left| \sum_{k \neq 0} \widehat{G}_{R,\delta}(k) \right|
\leq 2 \, C(m) \, \sum_{k \geq 1} e^{-\pi \, k^2 / R^2} \, (1+k^m).
\end{align*}
For a fixed value of $m$ and as $R \to 0$ the function $R \mapsto \sum_{k \geq 1} e^{-\pi \, k^2 / R^2} (1+k^m)$ is equivalent to $2 \, e^{-\pi/R^2}$, hence the conclusion.
\end{proof}
Now we are ready to prove the main result of this section.
\begin{proof}[Proof of Proposition \ref{keyprop2}]
By a non-asymptotic version of Stirling's formula (see pages 251-253 of \cite{Courseinmodernanalysis}) it follows that for any integer $l\ge 1$,
\begin{equation}\label{eqkfactorapprox}
-\frac{1}{360 l^3}\le \log( l!) - \left[l \log(l) - l +\frac{1}{2}\log(2\pi l)+\frac{1}{12l}\right]\le 0.
\end{equation}
From the definition of our model, we have
\[\log \mmu_{\tbeta_c,M}^{\mathrm{Potts}}(s_1,s_2,s_3)=\const(M) + \tbeta_c M (s_1^2+s_2^2+s_3^2)+\log\left(\frac{M!}{(Ms_1)! (Ms_2)! (Ms_3)!} \right).\]
A simple argument shows that $\mmu_{\tbeta_c,M}^{\mathrm{Potts}}(s_i=0)$ is exponentially small in $M$ for $i=1, 2, 3$. Therefore these terms do not affect the claim of the proposition, and thus they can be omitted.
Using \eqref{eqkfactorapprox}, we have that for every $s\in \Omega^{S}$ satisfying that $s_i\ne 0$ for $i=1,2,3$, 
\begin{align*}
\log \mmu_{\tbeta_c,M}^{\mathrm{Potts}}(s_1,s_2,s_3)&=\const(M) +  M \sum_{i=1}^{3}(\tbeta_c s_i^2 - s_i\log(s_i))\\
&- \frac{1}{2}\sum_{i=1}^{3}\log(s_i)-\frac{1}{12M} \sum_{i=1}^{3}\frac{1}{s_i} + E_0 -E_1-E_2-E_3,
\end{align*}
where $0\le E_0\le \frac{1}{360M^3}$, and $0\le E_i\le \frac{1}{360s_i^3 M^3}$ for $i=1,2,3$.
It is easy to see that for $s\in \Omega^{S}, s_i\ne 0$ for $i=1,2,3$, we have
\[\left|- \frac{1}{2}\sum_{i=1}^{3}\log(s_i)-\frac{1}{12M} \sum_{i=1}^{3}\frac{1}{s_i} + E_0 -E_1-E_2-E_3\right|\le \frac{1}{3}+\frac{3}{2}\log(M).\]
Using this and Lemma \ref{lemmaL} it follows that 
\begin{equation}\label{noncentralpartsboundineq}\mmu_{\tbeta_c,M}^{\mathrm{Potts}}\left(\{\sigma\in \Omega: d_C(s(\sigma))>\frac{\log(M)}{\sqrt{M}}\right)\le \frac{1}{M^2}
\end{equation}
for $M$ larger than some absolute constant. This can be neglected since it is $\OO\left(\frac{\log(M)^{3/2}}{M^{3/2}}\right)$. Therefore we will focus on the regions that are less than $\frac{\log(M)}{\sqrt{M}}$ distance away from one of the centers. Let
\begin{equation}\label{LMdef}L_{M}(s):=M \sum_{i=1}^{3}(\tbeta_c s_i^2 - s_i\log(s_i))
- \frac{1}{2}\sum_{i=1}^{3}\log(s_i)-\frac{1}{12M} \sum_{i=1}^{3}\frac{1}{s_i}, \end{equation}
then for $d_C(s)\le \frac{\log(M)}{\sqrt{M}}$, we have
$\log \mmu_{\tbeta_c,M}^{\mathrm{Potts}}(s_1,s_2,s_3)=\const(M)+ L_M(s)+\OO(\frac{1}{M^3})$. The error term can be neglected as previously, and thus we only need to analyse the function $L_M(s)$. 
First, we will consider $C_4$. By substituting $s_3=1-s_1-s_2$ and using Taylor's expansion with remainder term of $L_M(s)$ around $C_4$ in $s_1$ and $s_2$, we obtain that for $d(s,C_4)\le \frac{\log(M)}{\sqrt{M}}$,
\begin{align*}
&L_M(s)=M\left(\frac{\tbeta_c}{3}+\log(3)\right)+\frac{3}{2}\log(3)-\frac{3}{4M}\\
&-\left((3-2\tbeta_c)M-\frac{9}{2}\right)\cdot \left(\left(s_1-\frac{1}{3}\right)^2+\left(s_2-\frac{1}{3}\right)^2+\left(s_1-\frac{1}{3}\right)\left(s_2-\frac{1}{3}\right)\right)\\
&-\frac{9}{2}M \left(\left(s_1-\frac{1}{3}\right)^2\left(s_2-\frac{1}{3}\right)+\left(s_1-\frac{1}{3}\right)\left(s_2-\frac{1}{3}\right)^2\right)\\
&-\frac{9}{2}M \Bigg[\left(s_1-\frac{1}{3}\right)^4+\left(s_2-\frac{1}{3}\right)^4 + 2 \left(s_1-\frac{1}{3}\right)^3\left(s_2-\frac{1}{3}\right)\\
&+2\left(s_1-\frac{1}{3}\right)\left(s_2-\frac{1}{3}\right)^3+3\left(s_1-\frac{1}{3}\right)^2\left(s_2-\frac{1}{3}\right)^2\Bigg]+\OO\left(\frac{\log(M)^5}{M^{\frac{3}{2}}}\right).
\end{align*}
By change of variables $\ts_1:=s_1+s_2-\frac{2}{3}$, $\ts_2:=s_1-s_2$, we have
\begin{align*}
&\exp(L_M(s))=\exp\left(M\left(\frac{\tbeta_c}{3}+\log(3)\right)+\frac{3}{2}\log(3)-\frac{3}{4M}\right)
\cdot \exp\left(-M(3-2\tbeta_c)\left[\frac{3}{4}\ts_1^2+\frac{1}{4}\ts_2^2\right]\right)\cdot\\
&\cdot \Bigg[1+\frac{9}{2}\left[\frac{3}{4}\ts_1^2+\frac{1}{4}\ts_2^2\right]+\frac{9}{8}M[\ts_1\ts_2^2-\ts_1^3]-\frac{9}{32}M(9\ts_1^4+3\ts_2^4+4\ts_1^2\ts_2^2)\\
&+\frac{81}{32}M^2(\ts_1^6+4\ts_2^4\ts_1^2-2\ts_1^4\ts_2^2)+
\OO\left(\frac{\log(M)^5}{M^{\frac{3}{2}}}\right)\Bigg].
\end{align*}
Now it is easy to see that
\begin{align*}&\sum_{s: d(s,C_4)\le \frac{\log(M)}{\sqrt{M}}}\exp(L_M(s))=
\sum_{\ts_1\in \{-\frac{2}{3}+\frac{k}{M}: k\in \Z\}}\sum_{\ts_2\in \{\frac{l}{M}: l\in \Z, (\ts_1+\frac{2}{3})M-l \text{ even}\}}\\
&\exp\left(M\left(\frac{\tbeta_c}{3}+\log(3)\right)+\frac{3}{2}\log(3)-\frac{3}{4M}\right)\cdot 
\exp\left(-M(3-2\tbeta_c)\left[\frac{3}{4}\ts_1^2+\frac{1}{4}\ts_2^2\right]\right)\\
&\cdot \Bigg[1+\frac{9}{2}\left[\frac{3}{4}\ts_1^2+\frac{1}{4}\ts_2^2\right]+\frac{9}{8}M[\ts_1\ts_2^2-\ts_1^3]-\frac{9}{32}M(9\ts_1^4+3\ts_2^4+4\ts_1^2\ts_2^2)
\\&+
\frac{81}{32}M^2(\ts_1^6+4\ts_2^4\ts_1^2-2\ts_1^4\ts_2^2)\Bigg]\cdot \left(1+
\OO\left(\frac{\log(M)^5}{M^{\frac{3}{2}}}\right)\right).
\end{align*}
We can decompose each term of the above sum as a product of a sum in $\ts_1$ and another sum in $\ts_2$. These can be then approximated as integrals by Lemma \ref{lemmaRiemannapprox}, leading to
\begin{align*}
&\sum_{s: d(s,C_4)\le \frac{\log(M)}{\sqrt{M}}}\exp(L_M(s))=
\exp\left(M\left(\frac{\tbeta_c}{3}+\log(3)\right)+\frac{3}{2}\log(3)-\frac{3}{4M}\right)\\
&\cdot \Bigg( \frac{M\cdot 2\pi/\sqrt{3}}{3-4\log(2)}
+ \frac{\sqrt{3}\pi (99-64\log(2)+96\log(2)^2)}{2(3-4\log(2))^4}\Bigg) \left(1+
\OO\left(\frac{\log(M)^5}{M^{\frac{3}{2}}}\right)\right)\\
&=\exp\left(M\left(\frac{\tbeta_c}{3}+\log(3)\right)\right)\cdot M\cdot \left(c_1 +\frac{c_2}{M}+\OO\left(\frac{\log(M)^5}{M^{\frac{3}{2}}}\right)\right),
\end{align*}
for some absolute constants $c_1, c_2\in \R$, $c_1>0$.

By a similar argument, one can show that for $1\le i\le 3$, we also have
\begin{align*}
&\sum_{s: d(s,C_i)\le \frac{\log(M)}{\sqrt{M}}}\exp(L_M(s))\\
&=\exp\left(M\left(\frac{\tbeta_c}{3}+\log(3)\right)\right)\cdot M\cdot \left(c_1' +\frac{c_2'}{M}+\OO\left(\frac{\log(M)^5}{M^{\frac{3}{2}}}\right)\right),
\end{align*}
for some absolute constants $c_1', c_2'\in \R$, $c_1'>0$. By taking the ratio of these sums for the case when $M=j$, using \eqref{noncentralpartsboundineq}, we obtain that for $j\ge 1$, $1\le i\le 4$, the probability of the modes equals
\begin{equation}\label{eqFjprob}\mmu_{j}\left(F_j^{(i)}\right)=w_i + \frac{w_i'}{j}+ \OO\left(\frac{\log(j)^5}{j^{\frac{3}{2}}}\right),\end{equation}
for some absolute constants $w_i, w_i'\in \R$, $w_i>0$. 

For $j=0$ and $j=1$, we only have one mode, so $B_{0,1}=B_{1,2}=1$. For $j\ge 2$, the growth-within-mode constants $B_{j,j+1}$ can be written as
\[B_{j,j+1}=\max_{i\le 4}\frac{\mmu_{j+1}\left(F_j^{(i)}\right)}{\mmu_{j}\left(F_j^{(i)}\right)}\le \max_{i\le 4}\frac{\mmu_{j+1}\left(F_j^{(i)}\right)}{\mmu_{j+1}\left(F_{j+1}^{(i)}\right)}\cdot \max_{i\le 4}\frac{\mmu_{j+1}\left(F_{j+1}^{(i)}\right)}{\mmu_{j}\left(F_j^{(i)}\right)}.\]
Since $F_{j+1}^{(i)}$ and $F_j^{(i)}$ only differ on the borders between the modes, by Lemma \ref{lemmaL} is follows that $\left|\max_{i\le 4}\frac{\mmu_{j+1}\left(F_j^{(i)}\right)}{\mmu_{j+1}\left(F_{j+1}^{(i)}\right)} - 1\right|$ decreases exponentially fast in $j$. Moreover, from \eqref{eqFjprob}, it follows that
$\max_{i\le 4}\frac{\mmu_{j+1}\left(F_{j+1}^{(i)}\right)}{\mmu_{j}\left(F_j^{(i)}\right)}=1+\OO\left(\frac{\log(j)^5}{j^{\frac{3}{2}}}\right)$. Thus for any $j\ge 2$,
\[B_{j,j+1}=1+\OO\left(\frac{\log(j)^5}{j^{\frac{3}{2}}}\right),\]
and the claim of the proposition follows.
\end{proof}

\subsection{Proof of Lemma \ref{dtvLambdailemma}}\label{secproofdtvLambdailemma}
For any two distributions $\bm{\eta}, \bm{\nu}$ on some finite set $W$, we have 
\[\dtv(\bm{\eta},\bm{\nu})=\frac{1}{2}\sum_{x\in W}|\bm{\eta}(x)-\bm{\nu}(x)|=\sum_{x\in W}(\bm{\eta}(x)-\bm{\nu}(x))_+.\]
Based on the definition of $\mmu_{M|\tilde{\Lambda}^{(i)}}$ and $\mmu_{M}^{(i)}$, it follows that
$\mmu_{M}^{(i)}(x)<\mmu_{M|\tilde{\Lambda}^{(i)}}(x)$ for every $x\in \tilde{\Lambda}^{(i)}$, and thus 
\[\dtv\left(\mmu_{M}^{(i)},\mmu_{M|\tilde{\Lambda}^{(i)}}\right)= \mmu_{M}^{(i)}\left(F_M^{(i)}\setminus \tilde{\Lambda}^{(i)}\right)=
\frac{\mmu_{M}\left(F_M^{(i)}\setminus \tilde{\Lambda}^{(i)}\right)}{\mmu_{M}\left(F_M^{(i)}\right)}.\]
The claims of the Lemma now follow from \eqref{noncentralpartsboundineq} and \eqref{eqFjprob}.

%\bibliographystyle{plainnat}
%\bibliography{References}

\end{document}